\documentclass{amsart}

\usepackage[all]{xy}

\newfont{\cyr}{wncyr10}

\newcommand{\LL}{\mathbf{L}}

\newcommand{\HH}{\mathrm{H}}
\newcommand{\BB}{\mathrm{B}}
\newcommand{\ZZ}{\mathrm{Z}}

\theoremstyle{plain}
\newtheorem{thm}{Theorem}[section]
\newtheorem{prop}[thm]{Proposition}
\newtheorem{lemma}[thm]{Lemma}

\theoremstyle{definition}
\newtheorem{dfn}[thm]{Definition}

\newtheorem{hypo}{Hypothesis}

\theoremstyle{remark}
\newtheorem{rem}[thm]{Remark}

\begin{document}

\title[Obstructions]{Obstructions  for Deformations of Complexes}
\author[F. Bleher]{Frauke M. Bleher*}
\thanks{*Supported by 
NSA Grant H98230-06-1-0021 and NSF Grant  DMS0651332.}
\address{F.B.: Department of Mathematics\\University of Iowa\\
Iowa City, IA 52242-1419, U.S.A.}
\email{fbleher@math.uiowa.edu}

\author[T. Chinburg]{Ted Chinburg**}
\thanks{**Supported  by NSF Grants  DMS0500106
and DMS0801030. }
\address{T.C.: Department of Mathematics\\University of
Pennsylvania\\Philadelphia, PA
19104-6395, U.S.A.}
\email{ted@math.upenn.edu}

\date{September 10, 2010}

\begin{abstract}
We develop two approaches to obstruction theory for 
deformations of 
derived isomorphism classes of complexes $Z^\bullet$ of modules for a profinite group $G$ over a complete local Noetherian ring $A$ of positive residue characteristic $\ell$.  
\end{abstract}

\maketitle

\setcounter{tocdepth}{1}
\tableofcontents

%%%%%%%%%%%%%%%%%%%%%%%%%%%%%%%%%%%%%%%%%%%%%%%%%%%%%%%%%%%%%%%%%%%%%%%%%%%
%% Introduction
%%%%%%%%%%%%%%%%%%%%%%%%%%%%%%%%%%%%%%%%%%%%%%%%%%%%%%%%%%%%%%%%%%%%%%%%%%%

\section{Introduction}
\label{s:intro}
\setcounter{equation}{0}

Two basic tools of deformation theory are obstructions and parameterizations of infinitesimal deformations.  Obstructions determine when an object has an infinitesimal deformation.  When such an obstruction vanishes, one would like to parameterize all such infinitesimal deformations. In this paper we develop these tools in the context of
deforming derived isomorphism classes of complexes $Z^\bullet$ of modules for a profinite group $G$ over a complete local Noetherian ring $A$ 
having a fixed residue field $k$ of positive characteristic $\ell$.

The infinitesimal
deformation problem we consider has to do with lifting the isomorphism class
of $Z^\bullet$ in the derived category $D^-(A[[G]])$
of bounded above complexes of pseudocompact $A[[G]]$-modules to a class in $D^-(A'[[G]])$
when $A' \to A$ is a surjection of complete local Noetherian rings
whose kernel $J$ has square $0$.  The precise definition of the
deformation functor we consider is given in \S  \ref{s:lifts}.  

We give two different
approaches to obstruction theory.
The first, more naive, method proceeds by first replacing $Z^\bullet$ by a bounded
above complex of topologically free pseudocompact $A[[G]]$-modules. One can then separately lift each term 
of  $Z^\bullet$  to an $A'[[G]]$-module.  By considering the obstruction to 
lifting the boundaries of $Z^\bullet$ so as to obtain a complex of $A'[[G]]$-modules,
one arrives at a lifting obstruction $\omega(Z^\bullet,A')$ in $\mathrm{Ext}^2_{D^-(A[[G]])}(Z^\bullet,J\hat{\otimes}^{\LL}_{A} Z^\bullet)$.  Here $\hat{\otimes}^{\LL}$
is the left derived tensor product discussed in Remark \ref{rem:leftderivedtensor}. 

The second method uses a construction of Gabber and a suggestion
of Illusie.  This interprets the obstruction to
lifting $Z^\bullet$ as the image of a certain canonical
element under a boundary map in a spectral sequence which
computes $\mathrm{Ext}$ groups over $D^{-}(A'[[G]])$ via  $\mathrm{Ext}$ groups over $D^{-}(A[[G]])$ and $\mathrm{Tor}^{A'}$ complexes. We will describe
this in more detail below.

When the lifting obstruction vanishes, each of the two above methods describes all local
isomorphism classes of lifts of $Z^\bullet $ over $A'$  as a principal homogeneous space for $\mathrm{Ext}^1_{D^-(A[[G]])}(Z^\bullet,J\hat{\otimes}^{\LL}_{A} Z^\bullet)$.
 The spectral
sequence method has the advantage of identifying this principal homogeneous
space as a particular coset of 
$\mathrm{Ext}^1_{D^-(A[[G]])}(Z^\bullet,J\hat{\otimes}^{\LL}_{A} Z^\bullet)$ inside
$\mathrm{Ext}^1_{D^-(A'[[G]])}(Z^\bullet,J\hat{\otimes}^{\LL}_{A} Z^\bullet)$.  This 
identifies the local deformation functor as a subfunctor of a functor defined
by  $\mathrm{Ext}^1$ groups
over $D^{-}(A'[[G]])$.  
 The spectral sequence also gives a natural
filtration of $\mathrm{Ext}^1_{D^-(A'[[G]])}(Z^\bullet,J\hat{\otimes}^{\LL}_{A} Z^\bullet)$.
We obtain an interpretation of the last two terms in this filtration via exact sequences
of complexes which satisfy additional conditions.  

The spectral sequence we study is  
\begin{equation}
\label{eq:nightmare}
E_2^{p,q}=\mathrm{Ext}^p_{D^-(A[[G]])}(\mathcal{H}^{-q}(A \hat{\otimes}_{A'}^{\LL} Z^\bullet),
J\hat{\otimes}_A^{\LL}Z^\bullet)\;\Longrightarrow\;
\mathrm{Ext}^{p+q}_{D^-(A'[[G]])}(Z^\bullet, J\hat{\otimes}^{\LL}_AZ^{\bullet}).
\end{equation}
Here $\mathcal{H}^{-q}(A \hat{\otimes}_{A'}^{\LL} Z^\bullet)$ is a Tor complex 
whose $j^{\mathrm{th}}$ term is $\mathrm{Tor}_q^{A'}(Z^j,A)$ (see
Definition \ref{def:spectralseqA1}).
We will show in Theorem \ref{thm:obstructions} that the lifting obstruction  
$\omega(Z^\bullet,A')$ is the image under 
$d_2^{0,1}:E_2^{0,1} \to E_2^{2,0}$
of a canonical element $\iota$ in $E_2^{0,1}$.  In Theorem \ref{thm:bigobstructionthm}
(see also Lemmas \ref{lem:gabberconstruct} and \ref{lem:gabberlemma1}),
Gabber's construction will be shown to arise from the exact sequence of low degree terms
\begin{equation}
\label{eq:bignightmare}
0 \to E_\infty^{1,0} \to F_{II}^0 \to
E_2^{0,1}/W_2^{0,1} \xrightarrow{\overline{d_2^{0,1}}} E_2^{2,0}
\end{equation}
where $F_{II}^0  = F_{II}^0\, \mathrm{Ext}^1_{D^-(A'[[G]])}(Z^\bullet,J\hat{\otimes}^{\LL}_{A} Z^\bullet)$
 is the second to last term in the second filtration of the total cohomology of a bicomplex whose 
first total cohomology group is 
$\mathrm{Ext}^1_{D^-(A'[[G]])}(Z^\bullet,J\hat{\otimes}^{\LL}_{A} Z^\bullet)$ 
and $E_\infty^{0,1}=\mathrm{Ker}(d_2^{0,1})/W_2^{0,1}$ (see Definition \ref{def:spectralseqB1}).
We will  interpret $F_{II}^0$ as the set of extension classes arising from short
exact sequences of bounded above complexes of pseudocompact $A'[[G]]$-modules
\begin{equation}
\label{eq:bigsleep}
0 \to X^\bullet \to Y^\bullet \to Z^\bullet \to 0
\end{equation}
in which $X^\bullet$ is annihilated by $J$ and isomorphic to $J\hat{\otimes}^{\LL}_{A} Z^\bullet$
in $D^-(A[[G]])$. 
We will show in Lemma \ref{lem:extra} that if $(Z^\bullet,\zeta)$ has a lift over $A'$, then the 
local isomorphism class of every lift of $(Z^\bullet,\zeta)$ over $A'$ contains
a lift $(Y^\bullet,\upsilon)$ such that $Y^\bullet$ occurs as the middle term of a short exact sequence  
of the form $(\ref{eq:bigsleep})$.
We will show  in Theorem \ref{thm:obstructions} 
that if a lift of $(Z^\bullet,\zeta)$ over $A'$ exists, then the set of all local isomorphism classes 
of such lifts is in bijection with the full preimage of $\iota+W_2^{0,1}$ 
under the map $F_{II}^0\to E_2^{0,1}/W_2^{0,1}$ in $(\ref{eq:bignightmare})$.
This proves that the set of all local isomorphism classes of such lifts is a principal homogeneous 
space for $E_{\infty}^{1,0}$ and it gives a description of the  local isomorphism classes of
lifts of $(Z^\bullet,\zeta)$ over $A'$ in terms of classes in $F_{II}^0\subset
\mathrm{Ext}^1_{D^-(A'[[G]])}(Z^\bullet, J\hat{\otimes}^{\LL}_AZ^{\bullet})$.
Moreover, if a lift of $(Z^\bullet,\zeta)$ over $A'$ exists,  we will show that
$E_2^{p,0} = E_\infty^{p,0}$ for all $p$.  This partial degeneration is stronger than what is implied by
the naive method, which deals only with the case $p = 1$.

We now describe the sections of this paper.

In \S \ref{s:lifts} we recall the definitions and notations needed to state the main
result of \cite{bcderived} 
concerning the existence of versal and universal deformations 
of derived isomorphism classes 
of bounded complexes $V^\bullet$ in $D^-(k[[G]])$.
When $V^\bullet$ has only one non-zero term, this is the deformation
theory of continuous $G$-modules developed by Mazur in \cite{maz1}
using work of Schlessinger in  \cite{Sch}.
We also define local isomorphism classes of lifts over $A'$ of 
complexes $Z^\bullet$ in $D^-(A[[G]])$ 
relative to a surjection of complete local Noetherian rings $A'\to A$ with residue field $k$
having a square zero kernel.

The naive approach to obstruction theory is given in \S \ref{ss:naive}.
An outline of the spectral sequence approach, beginning with the
case of modules rather than complexes, is given in \S \ref{ss:spectralseqroad}.
The details of this approach for complexes are developed in \S  \ref{ss:spectralseq} - \S \ref{ss:aut}.  
The two methods are compared in \S \ref{ss:compare}.

The results of this paper are used in \cite{finiteness} to study a new finiteness
problem concerning deformations of arithmetically defined Galois
modules.  The particular result needed in \cite{finiteness} is Proposition \ref{prop:prop},
which shows that to determine versal deformations,
one can take the quotient of $G$ by any closed normal pro-prime-to-$\ell$ group which
acts trivially on $V^\bullet$ where $\ell$ is the characteristic of $k$.

\medbreak
\noindent {\bf Acknowledgments:}
The authors would like to thank
Ofer Gabber for explaining  his approach to obstruction theory and
for many helpful comments. 
The authors would also like to thank Luc Illusie for many valuable discussions without
which this paper would not have been possible. It was Illusie's idea to ask  Gabber
about obstructions, and he also suggested 
the idea of formulating Gabber's obstruction theory
in terms of spectral sequences.
The authors would also like to thank the Banff International Research Station for support 
during the preparation of part of this paper. 
\medbreak

%%%%%%%%%%%%%%%%%%%%%%%%%%%%%%%%%%%%%%%%%%%%%%%%%%%%%%%%%%%%%%%%%%%%%%%%%%%
%% Quasi-lifts and deformation functors
%%%%%%%%%%%%%%%%%%%%%%%%%%%%%%%%%%%%%%%%%%%%%%%%%%%%%%%%%%%%%%%%%%%%%%%%%%%

\section{Quasi-lifts and deformation functors}
\label{s:lifts}
\setcounter{equation}{0}

Let $G$ be a profinite group, let $k$ be a field of positive characteristic $\ell$,
and let $W$ be a complete local commutative Noetherian ring with residue field $k$.
Define
$\hat{\mathcal{C}}$ to be the category of complete local commutative Noetherian 
$W$-algebras
with residue field $k$. The morphisms in $\hat{\mathcal{C}}$ are 
continuous $W$-algebra
homomorphisms that induce the identity on $k$.
Let $\mathcal{C}$ be the subcategory of  Artinian objects in
$\hat{\mathcal{C}}$.
If $R \in \mathrm{Ob}(\hat{\mathcal{C}})$, let $R[[G]]$ be the completed group algebra of the usual
abstract group algebra $R[G]$ of $G$ over  $R$, i.e. $R[[G]]$ is the projective limit of the ordinary group algebras $R[G/U]$ as $U$ ranges over the open normal subgroups of $G$. 
 
\begin{dfn}
\label{def:pseudocompact}
A topological ring $\Lambda$ is called a \emph{pseudocompact ring} if $\Lambda$ is complete 
and Hausdorff and admits a basis of open neighborhoods of $0$ consisting of two-sided ideals $J$
for which $\Lambda/J$ is an Artinian ring.

Suppose $\Lambda$ is a pseudocompact ring. A complete Hausdorff topological 
$\Lambda$-module $M$ is said to be a \emph{pseudocompact $\Lambda$-module}
if $M$ has a basis of open neighborhoods  of $0$ consisting of submodules $N$ for which
$M/N$ has finite length as $\Lambda$-module. We denote by $\mathrm{PCMod}(\Lambda)$ 
the category of pseudocompact $\Lambda$-modules. (If not stated otherwise, our modules are
left modules.)

A pseudocompact $\Lambda$-module $M$ is said to be \emph{topologically free} on a set
$X=\{x_i\}_{i\in I}$ if $M$ is isomorphic to the product of a family $(\Lambda_i)_{i\in I}$ where
$\Lambda_i=\Lambda$ for all $i$.

Suppose $R$ is a commutative pseudocompact ring. A complete Hausdorff topological ring
$\Lambda$ is called a \emph{pseudocompact $R$-algebra} if $\Lambda$ is an $R$-algebra
in the usual sense, and if $\Lambda$ admits a basis of open neighborhoods of $0$ consisting of 
two-sided ideals $J$ for which $\Lambda/J$ has finite length as $R$-module. 

Suppose $\Lambda$ is a pseudocompact $R$-algebra, and let $\hat{\otimes}_{\Lambda}$ denote 
the completed tensor product in the category $\mathrm{PCMod}(\Lambda)$ (see \cite[\S 2]{brumer}).
If $M$ is a right (resp. left) pseudocompact $\Lambda$-module, then $M\hat{\otimes}_\Lambda- $ 
(resp. $-\hat{\otimes}_\Lambda M$) is a right exact functor.
Moreover,  $M$ is said to be  
\emph{topologically flat}, if the functor $M\hat{\otimes}_{\Lambda}-$ (resp. $-\hat{\otimes}_\Lambda M$)
is exact.
\end{dfn}

\begin{rem}
\label{rem:pseudocompact}
Pseudocompact rings, algebras and modules have been studied, for example, in 
\cite{ga1,ga2,brumer}. The following statements can be found in these references.
Suppose $\Lambda$ is a pseudocompact ring. 
\begin{enumerate}
\item[(i)] The ring $\Lambda$ is the projective limit 
of Artinian quotient rings having the discrete topology.
A $\Lambda$-module is pseudocompact if and only if it is the projective limit of 
$\Lambda$-modules of finite length having the discrete topology. 
The category $\mathrm{PCMod}(\Lambda)$ is an abelian
category with exact projective limits. 
\item[(ii)]  Every topologically free pseudocompact $\Lambda$-module is a projective object in 
$\mathrm{PCMod}(\Lambda)$, and every pseudocompact
$\Lambda$-module is the quotient of a topologically free $\Lambda$-module. Hence
$\mathrm{PCMod}(\Lambda)$ has enough projective objects. 
\item[(iii)] Every pseudocompact $R$-algebra is a pseudocompact ring, and a module over 
a pseudocompact $R$-algebra has finite length if and only if it has finite length as $R$-module.
\item[(iv)] Suppose $\Lambda$ is a pseudocompact $R$-algebra, and $M$ and $N$ are 
pseudocompact $\Lambda$-modules. Then we define the right derived functors
$\mathrm{Ext}^n_{\Lambda}(M,N)$ by using a projective resolution of $M$. 
\item[(v)]  Suppose $R\in\mathrm{Ob}(\hat{\mathcal{C}})$. Then $R$ is a pseudocompact 
ring, and $R[[G]]$ is a pseudocompact $R$-algebra. 
\end{enumerate}
\end{rem}

\begin{rem}
\label{rem:extrafree}
Let $R$ be an object in $\hat{\mathcal{C}}$ and let $m_R$ be its maximal ideal. Suppose 
$[(R/m_R^i)X_i]$ is an abstractly free $(R/m_R^i)$-module on the finite topological space $X_i$
for all $i$,
and that $\{X_i\}_i$ forms an inverse system. Define $X=\displaystyle \lim_{\stackrel{
\longleftarrow}{i}} X_i$ and $R[[X]] = \displaystyle \lim_{\stackrel{\longleftarrow}{i}} \,
[(R/m_R^i) X_i]$. Then
$R[[X]]$ is a topologically free pseudocompact $R$-module on $X$. In particular,
every topologically free pseudocompact $R[[G]]$-module is a topologically free 
pseudocompact $R$-module. 
\end{rem}

\begin{rem}
\label{rem:dumber}
Suppose $R$ is an object in $\hat{\mathcal{C}}$ and $\Lambda=R$ or $R[[G]]$. Let $M$ be a pseudocompact right (resp. left) $\Lambda$-module. 
\begin{enumerate}
\item[(i)]
If $M$ is finitely generated as a pseudocompact $\Lambda$-module, it follows from
\cite[Lemma 2.1(ii)]{brumer} that the functors
$M \otimes_\Lambda -$ and $M\hat{\otimes}_\Lambda -$ (resp. $-\otimes_\Lambda M$ and 
$-\hat{\otimes}_\Lambda M$) are naturally isomorphic.
\item[(ii)]
By \cite[Lemma 2.1(iii)]{brumer} and \cite[Prop. 3.1]{brumer}, $M$ is topologically flat if and only 
if $M$ is projective.
\item[(iii)]
If $\Lambda=R$ and $M$ is a pseudocompact $R$-module,
it follows from \cite[Proof of Prop. 0.3.7]{ga2} and \cite[Cor. 0.3.8]{ga2} that $M$ is 
topologically flat if and only if $M$ is topologically free if and only if $M$ is abstractly flat.
In particular, if $R$ is Artinian, a pseudocompact $R$-module is topologically flat  if and only
if it is abstractly free.
\end{enumerate}
\end{rem}

If $\Lambda$ is a pseudocompact ring, let $C^-(\Lambda)$
be the abelian category of complexes of pseudocompact $\Lambda$-modules that are bounded above, let $K^-(\Lambda)$ be the homotopy category of $C^-(\Lambda)$, and
let $D^-(\Lambda)$ be the derived category of $K^-(\Lambda)$. 
Let $[1]$ denote the translation functor on $C^-(\Lambda)$ (resp. $K^-(\Lambda)$, 
resp. $D^-(\Lambda)$), 
i.e. $[1]$ shifts complexes
one place to the left and changes the sign of the differential.
Note that a homomorphism in $C^-(\Lambda)$
is a quasi-isomorphism if and only if the induced homomorphisms on 
all the cohomology groups are bijective.

\begin{rem}
\label{rem:leftderivedtensor}
Let  $X^\bullet, Y^\bullet\in \mathrm{Ob}(K^-(R[[G]]))$ and consider the double complex $K^{\bullet,\bullet}$ of pseudocompact $R[[G]]$-modules with
$K^{p,q}=(X^p\hat{\otimes}_RY^q)$ and diagonal $G$-action. We define the total tensor product $X^\bullet\hat{\otimes}_R Y^\bullet$ to be the simple complex associated to $K^{\bullet,\bullet}$, i.e.
$$(X^\bullet\hat{\otimes}_RY^\bullet)^n=\bigoplus_{p+q=n}X^p\hat{\otimes}_RY^q$$
whose differential is $d(x\,\hat{\otimes}\, y)=d_X(x)\,\hat{\otimes} \,y + (-1)^x \,x\,\hat{\otimes} \,d_Y(y)$
for $x\,\hat{\otimes}\,y\in K^{p,q}$.
Since homotopies carry over the completed tensor product, we have a functor
$$\hat{\otimes}_R : K^-(R[[G]])\times K^-(R[[G]])\to K^-(R[[G]]).$$
Using \cite[Thm. 2.2 of Chap. 2 \S2]{verdier}, we see that there is a well-defined
left derived completed tensor product $\hat{\otimes}^{\LL}_R$. Moreover, if
$X^\bullet$ and $Y^\bullet$ are as above, then $X^\bullet\hat{\otimes}^{\LL}_R Y^\bullet$
may be computed in  $D^-(R[[G]])$ in the following way.  
Take a bounded above complex ${Y'}^\bullet$ of topologically flat pseudocompact $R[[G]]$-modules
with a quasi-isomorphism ${Y'}^\bullet\to Y^\bullet$ in $K^-(R[[G]])$. Then this quasi-isomorphism 
induces an isomorphism between
$X^\bullet\hat{\otimes}_R{Y'}^\bullet$ and $X^\bullet\hat{\otimes}^{\LL}_RY^\bullet$ in $D^-(R[[G]])$.
\end{rem}

\begin{dfn}
\label{def:fintor}
We will say that a complex $M^\bullet$ in $K^-(R[[G]])$
has \emph{finite pseudocompact $R$-tor dimension}, 
if there exists an integer $N$ such that for all pseudocompact
$R$-modules $S$, and for all integers $i<N$, ${\HH}^i(S\hat{\otimes}^{\LL}_R M^\bullet)=0$.
If we want to emphasize the integer $N$ in this definition, we say $M^\bullet$ has 
\emph{finite pseudocompact $R$-tor dimension at $N$}.
\end{dfn}

\begin{rem}
\label{rem:dumbdumb}
Suppose $M^\bullet$ is a complex in $K^-([[RG]])$ of topologically flat, hence topologically free, 
pseudocompact $R$-modules that has finite pseudocompact $R$-tor dimension
at $N$. Then the bounded complex ${M'}^\bullet$, which is
obtained from $M^\bullet$ by replacing $M^N$ by
${M'}^N=M^N/\delta^{N-1}(M^{N-1})$ and by setting ${M'}^i = 0$ if $i < N$,
is quasi-isomorphic to $M^\bullet$ and
has topologically free pseudocompact terms over $R$.
\end{rem}

\begin{hypo}
\label{hypo:fincoh}
Throughout this paper, we assume that $V^\bullet $ is a
complex in $D^-(k[[G]])$ that has  only finitely many non-zero cohomology
groups, all of which have finite $k$-dimension.
\end{hypo}

\begin{dfn}
\label{def:lifts}
A \emph{quasi-lift}
of $V^\bullet$ over an object $R$ of
$\hat{\mathcal{C}}$ is a pair $(M^\bullet,\phi)$ consisting of a complex
$M^\bullet$ in
$D^-(R[[G]])$ that has finite pseudocompact $R$-tor  dimension
together with an isomorphism
\hbox{$\phi: k \hat{\otimes}^{\LL}_R M^\bullet \to V^\bullet$} in $D^-(k[[G]])$.
Two
quasi-lifts $(M^\bullet, \phi)$ and $({M'}^\bullet,\phi')$ are
\emph{isomorphic} if there is an isomorphism
$f:M^\bullet \to {M'}^\bullet$ in $D^-(R[[G]])$ with
$\phi'\circ(k\hat{\otimes}^{\LL}f)=\phi$.
\end{dfn}

\begin{thm} 
\label{thm:derivedresult}  
Suppose that $\HH^i(V^\bullet) = 0$ unless $n_1 \le i \le n_2$.  Every quasi-lift of $V^\bullet$ over 
an object $R$ of $\hat{\mathcal{C}}$ is isomorphic to a quasi-lift $(P^\bullet, \phi)$ for
a complex $P^\bullet$ with the following properties:
\begin{enumerate}
\item[(i)] The terms of $P^\bullet$ are topologically free $R[[G]]$-modules.
\item[(ii)] The cohomology group $\HH^i(P^\bullet)$ is finitely generated 
as an abstract $R$-module for all $i$, and $\HH^i(P^\bullet) = 0$ unless $n_1 \le i \le n_2$. 
\item[(iii)]   One has $\HH^i(S\hat{\otimes}^{\LL}_RP^\bullet)=0$  for all pseudocompact $R$-modules 
$S$ unless $n_1 \le i \le n_2$.
\end{enumerate}
\end{thm}

\begin{proof}
Part (i)  follows from \cite[Lemma 2.9]{bcderived}. Assume now
that the terms of $P^\bullet$ are topologically free $R[[G]]$-modules, which means in particular
that the functors $-\hat{\otimes}^{\LL}_RP^\bullet$ and $-\hat{\otimes}_RP^\bullet$ are
naturally isomorphic. Let $m_R$ denote the maximal ideal of $R$, and let $n$ be an
arbitrary positive integer.
By \cite[Lemmas 3.1 and 3.8]{bcderived}, $\HH^i((R/m_R^n)\hat{\otimes}_RP^\bullet)=0$ for $i>n_2$ 
and $i<n_1$. Moreover, for $n_1\le i\le n_2$,
$\HH^i((R/m_R^n)\hat{\otimes}_RP^\bullet)$ is a subquotient of  an abstractly free $(R/m_R^n)$-module
of rank $d_i=\mathrm{dim}_k\,\HH^i(V^\bullet)$,
and $(R/m_R^n)\hat{\otimes}_RP^\bullet$ has finite pseudocompact $(R/m_R^n)$-tor dimension at
$N=n_1$. Since 
$P^\bullet\cong \displaystyle \lim_{\stackrel{\longleftarrow}{n}}\, (R/m_R^n)\hat{\otimes}_RP^\bullet$ 
and since by Remark \ref{rem:pseudocompact}(i), the category $\mathrm{PCMod}(R)$ has 
exact projective limits, it  follows that for all pseudocompact $R$-modules $S$
$$\HH^i(S\hat{\otimes}_RP^\bullet)= \lim_{\stackrel{\longleftarrow}{n}}\,\HH^i\left(
(S/m_R^nS)\hat{\otimes}_{R/m_R^n}\left((R/m_R^n)\hat{\otimes}_R P^\bullet\right)\right)$$ 
for all $i$. Hence Theorem \ref{thm:derivedresult} follows.
\end{proof}

\begin{dfn}
\label{def:functordef}
Let $\hat{F} = \hat{F}_{V^\bullet}:\hat{\mathcal{C}} \to \mathrm{Sets}$
be the functor which sends an object $R$ of $\hat{\mathcal{C}}$ to the set
$\hat{F}(R)$ of all isomorphism classes of quasi-lifts 
of $V^\bullet$ over $R$, and which sends
a morphism $\alpha:R\to R'$ in $\hat{\mathcal{C}}$ to the set map
$\hat{F}(R)\to \hat{F}(R')$ induced by $M^\bullet \mapsto R'\hat{\otimes}_{R,\alpha}^{\LL}
M^\bullet$. Let $F = F_{V^\bullet}$ be the restriction of $\hat{F}$ 
to the subcategory $\mathcal{C}$ of Artinian objects in $\hat{\mathcal{C}}$.

Let $k[\varepsilon]$, where $\varepsilon^2=0$, denote the ring of dual numbers over
$k$. The set $F(k[\varepsilon])$ is called the \emph{tangent space} to 
$F$, denoted by $t_{F}$. 
\end{dfn}

\begin{dfn}
\label{dfn:bndcoh}
A profinite group $G$ has \emph{finite pseudocompact cohomology},
if for each discrete $k[[G]]$-module $M$ of finite $k$-dimension,
and all integers $j$, the cohomology group ${\HH}^j(G,M)=\mathrm{Ext}^j_{k[[G]]}(k,M)$ 
has finite $k$-dimension.
\end{dfn}

\begin{thm}
\label{thm:bigthm} {\rm (\cite[Thm.  2.14]{bcderived})}
Suppose that $G$ has finite pseudocompact cohomology.
\begin{enumerate}
\item[(i)] 
The functor
$F$ has a pro-representable hull 
$R(G,V^\bullet)\in \mathrm{Ob}(\hat{\mathcal{C}})$ 
$($c.f. \cite[Def. 2.7]{Sch} and \cite[\S 1.2]{Maz}$)$, and 
the functor
$\hat{F}$ is continuous
$($c.f. \cite{Maz}$)$. 

\item[(ii)]
There is a $k$-vector space isomorphism $h: t_F \to
\mathrm{Ext}^1_{D^-(k[[G]])}(V^\bullet,V^\bullet)$.

\item[(iii)]
If $\mathrm{Hom}_{D^-(k[[G]])}(V^\bullet,V^\bullet)= k$, then $\hat{F}$ is represented
by $R(G,V^\bullet)$. 
\end{enumerate}
\end{thm}

\begin{rem}
\label{rem:newrem}
By Theorem \ref{thm:bigthm}(i), there exists 
a quasi-lift $(U(G,V^\bullet),\phi_U)$ of $V^\bullet$ over $R(G,V^\bullet)$ 
with the following property. For each $R\in \mathrm{Ob}(\hat{\mathcal{C}})$, the map
$\mathrm{Hom}_{\hat{\mathcal{C}}}(R(G,V^\bullet),R) \to \hat{F}(R)$ 
induced by $\alpha \mapsto R\hat{\otimes}^{\LL}_{R(G,V^\bullet),\alpha} U(G,V^\bullet)$ is surjective,
and this map is bijective if $R$ is the ring of dual numbers $k[\varepsilon]$ over $k$
where $\varepsilon^2=0$.

In general,
the isomorphism type of the pro-representable hull $R(G,V^\bullet)$ is
unique up to non-canonical isomorphism.
If $R(G,V^\bullet)$ represents $\hat{F}$, then $R(G,V^\bullet)$
is uniquely determined up to canonical isomorphism.
\end{rem}

\begin{dfn}
\label{def:newdef}
Using the notation of Theorem \ref{thm:bigthm} and Remark \ref{rem:newrem}, 
we call 
$R(G,V^\bullet)$
the \emph{versal deformation ring}
of $V^\bullet$ and  $(U(G,V^\bullet),\phi_U)$ 
a \emph{versal deformation} of $V^\bullet$.

If $R(G,V^\bullet)$ represents $\hat{F}$, then
$R(G,V^\bullet)$ will be
called the \emph{universal deformation ring} 
of $V^\bullet$ and $(U(G,V^\bullet),\phi_U)$  will be called a
\emph{universal deformation} of $V^\bullet$.
\end{dfn}

\begin{rem}
\label{rem:bigthm}
If $V^\bullet$ consists of a single module $V_0$ in dimension $0$,
the versal deformation ring $R(G,V^\bullet)$ coincides with the versal deformation
ring studied by Mazur in \cite{maz1,Maz}.
In this case,
Mazur assumed only that $G$ satisfies a certain finiteness condition
($\Phi_p$), which is equivalent
to the requirement that ${\HH}^1(G,M)$ have finite $k$-dimension for all discrete 
$k[[G]]$-modules $M$ of finite $k$-dimension.
Since the higher $G$-cohomology enters into determining lifts of
complexes $V^\bullet$
having more than one non-zero cohomology group, the condition that $G$
have finite pseudocompact
cohomology is the natural generalization of Mazur's finiteness condition
in this context.
\end{rem}

We also need to set up some notation concerning local deformation functors.

\begin{dfn}
\label{def:localdef}
Let $V^\bullet$ be as in Hypothesis \ref{hypo:fincoh}, let $A$ be in $\hat{\mathcal{C}}$,
and let $(Z^\bullet,\zeta)$ be a quasi-lift of $V^\bullet$ over $A$.
Let $A'\to A$ in $\hat{\mathcal{C}}$ be a surjective morphism in $\hat{\mathcal{C}}$ whose kernel 
is an ideal $J$ with $J^2=0$. 

A (\emph{local}) \emph{quasi-lift} of $(Z^\bullet,\zeta)$ over $A'$ is a pair $(Y^\bullet,\upsilon)$ 
consisting of
a complex $Y^\bullet$ in $D^-(A'[[G]])$ that has finite pseudocompact $A'$-tor dimension 
together with an isomorphism $\upsilon:A\hat{\otimes}^{\LL}_{A'}Y^\bullet\to Z^\bullet$ in $D^-(A[[G]])$.
Note that if $(Y^\bullet,\upsilon)$ is a quasi-lift of $(Z^\bullet,\zeta)$ over $A'$, then
$(Y^\bullet,\zeta\circ (k\hat{\otimes}^{\LL}\upsilon))$ is a quasi-lift of $V^\bullet$ over $A'$.

Two quasi-lifts $(Y^\bullet,\upsilon)$ and $({Y'}^\bullet,\upsilon')$ of $(Z^\bullet,\zeta)$ over $A'$
are said to be \emph{locally isomorphic} if there exists an isomorphism $f:Y^\bullet\to {Y'}^\bullet$
in $D^-(A'[[G]])$ with $\upsilon'\circ (A\hat{\otimes}^{\LL}f) = \upsilon$.
\end{dfn}

%%%%%%%%%%%%%%%%%%%%%%%%%%%%%%%%%%%%%%%%%%%%%%%%%%%%%%%%%%%%%%%%%%%%%%%%%%%
%% Obstructions
%%%%%%%%%%%%%%%%%%%%%%%%%%%%%%%%%%%%%%%%%%%%%%%%%%%%%%%%%%%%%%%%%%%%%%%%%%%

\section{Obstructions}
\label{s:obstruct}
\setcounter{equation}{0}

Let $V^\bullet$ be as in Hypothesis \ref{hypo:fincoh}, let $A$ be in $\hat{\mathcal{C}}$,
and let $(Z^\bullet,\zeta)$ be a quasi-lift of $V^\bullet$ over $A$.
Let $A'\to A$ in $\hat{\mathcal{C}}$ be a surjective morphism in $\hat{\mathcal{C}}$ whose kernel 
is an ideal $J$ with $J^2=0$. In this section, we 
develop the two different approaches described in the introduction to finding 
a lifting obstruction
$$\omega(Z^\bullet,A')\in 
\mathrm{Ext}^2_{D^-(A[[G]])}(Z^\bullet,J\hat{\otimes}^{\LL}_AZ^\bullet)$$ 
which vanishes if and only if $(Z^\bullet,\zeta)$ can be lifted to $A'$. 
The naive approach  is given in \S \ref{ss:naive} while the
spectral sequence approach is developed in \S \ref{ss:spectralseqroad} - \S \ref{ss:aut}.  
The two methods are compared in \S \ref{ss:compare}. More precisely,
we show that the lifting obstruction from either method can be obtained from the other
by composing with suitable automorphisms of $Z^\bullet$ and $J\hat{\otimes}^{\LL}_AZ^\bullet[2]$,
respectively, in $D^-(A[[G]])$.

Using the results from \S\ref{s:lifts}, we can make the following assumption concerning
$V^\bullet$ and $Z^\bullet$.

\begin{hypo}
\label{hypo:obstruct}
Assume $V^\bullet$ is as in Hypothesis \ref{hypo:fincoh} with $\HH^i(V^\bullet)=0$
unless $-p_0\le i\le -1$.
Suppose $0\to J\to A'\to A\to 0$ 
is an extension of objects $A',A$ in $\hat{\mathcal{C}}$ with $J^2=0$.
Let $B'=A'[[G]]$ and $B=A[[G]]$.

Let $(Z^\bullet,\zeta)$ be a quasi-lift of $V^\bullet$ over $A$. 
By Theorem \ref{thm:derivedresult} and Remark \ref{rem:dumbdumb}, we can make
the following assumptions:
The complex $Z^\bullet$ is a bounded complex of pseudocompact $B$-modules whose terms $Z^i$ are
zero unless $-p_0\le i\le -1$. The terms $Z^i$ are topologically flat, hence projective, 
pseudocompact $B$-modules for $i\neq -p_0$, and $Z^{-p_0}$ is topologically flat, hence topologically free, over $A$.
\end{hypo}

\begin{rem}
\label{rem:AG}
The functors $A\hat{\otimes}_{A'}-$ and 
$B\hat{\otimes}_{B'}-$ are naturally isomorphic functors $\mathrm{PCMod}(B')\to
\mathrm{PCMod}(B)$. 
Similarly to Remark \ref{rem:leftderivedtensor}, one obtains a 
well-defined left derived completed tensor product $B\hat{\otimes}_{B'}^{\LL}-$.
The functors $A\hat{\otimes}_{A'}^{\LL}-$ and 
$B\hat{\otimes}_{B'}^{\LL}-$ are naturally isomorphic functors  $D^-(B')\to D^-(B)$. 
\end{rem}

%%%%%%%%%%%%%%%%%%%%%%%%%%%%%%%%%%%%%%%%%%%%%%%%%%%
%%A naive approach
%%%%%%%%%%%%%%%%%%%%%%%%%%%%%%%%%%%%%%%%%%%%%%%%%%%

\subsection{A naive approach}
\label{ss:naive}

In this subsection we describe a naive approach to obstruction theory. 
We assume Hypothesis \ref{hypo:obstruct}. 
Let $(\tilde{Z}^\bullet,\tilde{\zeta})$ be a quasi-lift of $V^\bullet$ over $A$ that is isomorphic
to the quasi-lift $(Z^\bullet,\zeta)$ such that $\tilde{Z}^\bullet$ is concentrated in degrees 
$\le -1$ and all terms of $\tilde{Z}^\bullet$ are topologically free pseudocompact $B$-modules.
For each $j\in\mathbb{Z}$, let $Y^j$ be a topologically free pseudocompact $B'$-module 
which is a lift of $\tilde{Z}^j$ over $A'$ and let $a_Y^j:Y^j\to \tilde{Z}^j$ be 
the composition of the natural surjection $Y^j\to A\hat{\otimes}_{A'}Y^j$
followed by  $A\hat{\otimes}_{A'}Y^j\xrightarrow{\cong} \tilde{Z}^j$. Moreover,
let $c_Y^j:Y^j\to Y^{j+1}$ be a homomorphism of 
pseudocompact $B'$-modules such that $a_Y^{j+1}\circ c_Y^j = d_{\tilde{Z}}^j \circ a_Y^j$ for all $j$. 
In particular, $Y^j=0$ for $j\ge 0$, and  $c_Y^j=0$ for $j\ge -1$.
Note that $c_Y^{j+1}\circ c_Y^j$ may be non-zero so that $(Y^j,c_Y^j)_j$ is not necessarily a complex.
However, $(JY^j,c_{Y}^j\big|_{JY^j})_j$ defines a complex $JY^\bullet$ in $C^-(B)$
which is isomorphic to $J\hat{\otimes}_A\tilde{Z}^\bullet$ in $C^-(B)$.
For all $j\in\mathbb{Z}$, define $\tilde{\omega}^j:\tilde{Z}^j\to JY^{j+2}$ by
\begin{equation}
\label{eq:yuckyuck}
\tilde{\omega}^j(a_Y^j(y)) = c_Y^{j+1}(c_Y^j(y))
\end{equation} 
for all $y\in Y^j$.
Then $\tilde{\omega}\in\mathrm{Hom}_{C^-(B)}(\tilde{Z}^\bullet,JY^\bullet[2])$. 
Let $\omega_0(Z^\bullet,A')$ be the corresponding morphism in
$\mathrm{Ext}^2_{D^-(B)}(Z^\bullet,J\hat{\otimes}_AZ^\bullet)
\cong \mathrm{Hom}_{K^-(B)}(\tilde{Z}^\bullet,JY^\bullet[2])$.

We will show in \S\ref{ss:compare} that $\omega_0(Z^\bullet,A')$ is independent of choices by 
showing that $\omega_0(Z^\bullet,A')$ can be obtained from the lifting obstruction defined by a 
spectral sequence by composing with suitable automorphisms of $Z^\bullet$ and 
$J\hat{\otimes}_AZ^\bullet[2]$, respectively, in $D^-(B)$ (see Proposition \ref{prop:compare}). 

In particular, by using a fixed versal deformation of $V^\bullet$ over $R=R(G,V^\bullet)$
whose terms are topologically free pseudocompact $R[[G]]$-modules, we can assume that if 
there exists a quasi-lift of $(Z^\bullet,\zeta)$ over $A'$, then it is locally isomorphic to a 
quasi-lift $(\tilde{Y}^\bullet,\tilde{\upsilon})$ of $(Z^\bullet,\zeta)$ over $A'$
satisfying $\tilde{Y}^j=Y^j$ for all $j$. 

Since $\omega_0(Z^\bullet,A')=0$ in $D^-(B)$ if and only if $\tilde{\omega}$ is homotopic to zero 
in $C^-(B)$, we see the following. If there exists a quasi-lift $(\tilde{Y}^\bullet,\tilde{\upsilon})$ of 
$(Z^\bullet,\zeta)$ over $A'$ such that $\tilde{Y}^j=Y^j$ for all $j$, then the homotopy
$h^j:\tilde{Z}^j\to JY^{j+1}=J\tilde{Y}^{j+1}$ defined by $h^j\circ a_Y^j = c_Y^j - d_{\tilde{Y}}^j$
for all $j$ can be used to show that $\tilde{\omega}=0$ in $K^-(B)$. On the other hand,
if $\tilde{\omega}$ is homotopic to zero in $C^-(B)$, then the corresponding homotopy 
can be used to correct the maps $c_Y^j$ to obtain a complex $(Y^\bullet,d_Y)$ in $C^-(B')$ which 
defines a quasi-lift of $(Z^\bullet,\zeta)$  over $A'$.

Suppose now that $\omega_0(Z^\bullet,A')=0$, and let $(Y_0^\bullet,\upsilon_0)$ and
$({Y'}^\bullet,\upsilon')$ be two quasi-lifts of $(Z^\bullet,\zeta)$ over $A'$.
As seen above, we can assume without loss of generality that $Y_0^j=Y^j={Y'}^j$  for all $j$.
For all $j\in\mathbb{Z}$, define $\tilde{\beta}_{Y'}^j:\tilde{Z}^j\to JY^{j+1}$ by
\begin{equation}
\label{eq:principali}
\tilde{\beta}_{Y'}^j(a_Y^j(y))= d_{Y'}^j(y) - d_{Y_0}^j(y)
\end{equation}
for $y\in Y_0^j=Y^j={Y'}^j$.
Then $\tilde{\beta}_{Y'}\in\mathrm{Hom}_{C^-(B)}(\tilde{Z}^\bullet,
JY^\bullet[1])$. Let $\beta_{Y'}$ be the corresponding morphism in
$\mathrm{Ext}^1_{D^-(B)}(Z^\bullet,J\hat{\otimes}_AZ^\bullet)
\cong \mathrm{Hom}_{K^-(B)}(\tilde{Z}^\bullet,JY^\bullet[1])$.

We will show later that this can be used to prove that the set of all local isomorphism classes of 
quasi-lifts of $(Z^\bullet,\zeta)$ over $A'$ is a principal homogeneous space for 
$\mathrm{Ext}^1_{D^-(B)}(Z^\bullet,J\hat{\otimes}_AZ^\bullet)$, by relating this to the 
corresponding result obtained from the spectral sequence method.
More precisely, we will show that if the local isomorphism classes of the quasi-lifts 
$(Y_0^\bullet,\upsilon_0)$ and $({Y'}^\bullet,\upsilon')$ of $(Z^\bullet,\zeta)$ over $A'$ 
correspond to the classes $\eta_0$ and $\eta'$, respectively,  in 
$\mathrm{Ext}^1_{D^-(B')}(Z^\bullet,J\hat{\otimes}_AZ^\bullet)$ by the spectral sequence method,
then the difference $\eta'-\eta_0$ in  $\mathrm{Ext}^1_{D^-(B')}(Z^\bullet,J\hat{\otimes}_AZ^\bullet)$
is uniquely determined by $\beta_{Y'}$ (see Proposition \ref{prop:compare}).

%%%%%%%%%%%%%%%%%%%%%%%%%%%%%%%%%%%%%%%%%%%%%%%%%%%
%%A roadmap to the spectral sequence method
%%%%%%%%%%%%%%%%%%%%%%%%%%%%%%%%%%%%%%%%%%%%%%%%%%%

\subsection{Outline of the spectral sequence approach}
\label{ss:spectralseqroad}

In this subsection we introduce the spectral sequence approach to obstruction
theory by discussing the case of modules and by then indicating what adjustments
must be made for complexes.  This method goes back to Illusie in \cite[\S 3.1]{Illusie}.
It requires more effort than the naive approach, but as indicated in the introduction,
it places the local lifting problem in the context of studying $\mathrm{Ext}^1$ groups.

Let $Z$ be a
pseudocompact $B$-module which is (abstractly) free and finitely generated over $A$. 
We have a convergent spectral sequence
\begin{equation}
\label{eq:spectralmodule}
E_2^{p,q}=\mathrm{Ext}_B^p(\mathrm{Tor}^{A'}_q(Z,A),J\hat{\otimes}_AZ)
\Longrightarrow \mathrm{Ext}_{B'}^{p+q}(Z, J\hat{\otimes}_AZ).
\end{equation}
This arises in the following way.  To find the groups $\mathrm{Tor}^{A'}_q(Z,A)$,
one chooses a resolution $P^\bullet$ of $Z$ by projective pseudocompact $B'$-modules.
Then $\mathrm{Tor}^{A'}_q(Z,A) = \HH^{-q}(A \hat{\otimes}_{A'} P^\bullet)$, and the group
$\mathrm{Ext}_{B'}^{p+q}(Z, J\hat{\otimes}_AZ)$ is the group $\HH^{p+q}(\mathrm{Hom}_{B'}(P^\bullet,J\hat{\otimes}_A Z))$.
The key observation is that since $J \hat{\otimes}_A Z$ is a $B$-module, the complex
$\mathrm{Hom}_{B'}(P^\bullet,J\hat{\otimes}_A Z)$ is canonically isomorphic to the complex 
$\mathrm{Hom}_{B}(A \hat{\otimes}_{A'} P^\bullet,J\hat{\otimes}_A Z)$.  A Cartan-Eilenberg
resolution $M^{\bullet,\bullet}$
of $A\hat{\otimes}_{A'}P^\bullet$ is a double complex of projective pseudocompact $B$-modules
which gives a resolution of each term of $A\hat{\otimes}_{A'}P^\bullet$ which is compatible
with boundary maps and has some additional splitting properties (see \cite[\S (11.7) of Chap. 0]{ega3}).  One arrives at a double complex $L^{\bullet,\bullet}$ of 
$B$-modules given by $L^{q,p}=\mathrm{Hom}_B(M^{-q,-p},J\hat{\otimes}_AZ)$ such
that $$\HH^{p+q}(\mathrm{Tot}(L^{\bullet,\bullet})) = \HH^{p+q}(\mathrm{Hom}_{B}(A \hat{\otimes}_{A'} P^\bullet,J\hat{\otimes}_A Z)) = \mathrm{Ext}_{B'}^{p+q}(Z, J\hat{\otimes}_AZ).$$
The spectral sequence (\ref{eq:spectralmodule})
is then the spectral sequence of  $L^{\bullet,\bullet}$ relative to the second filtration 
of the total complex $\mathrm{Tot}(L^{\bullet,\bullet})$. We obtain the following exact sequence of
low degree terms associated to the spectral sequence $(\ref{eq:spectralmodule})$:
\begin{equation}
\label{eq:lowmodule}
0\to E_2^{1,0} \to \mathrm{Ext}^1_{B'}(Z,J\hat{\otimes}_AZ) \to E_2^{0,1} \xrightarrow{d_2^{0,1}}
E_2^{2,0}
\end{equation}

We now
sketch Gabber's approach to realizing the obstruction to lifting $Z$ from $A$
to $A'$ via the spectral sequence (\ref{eq:spectralmodule}).  We can find an exact sequence
\begin{equation}
\label{eq:original}
0 \to T \xrightarrow{\delta} P^0 \xrightarrow{\epsilon} Z \to 0
\end{equation}
in which $P^0$ is a finitely generated projective pseudocompact $B'$-module.
Applying the functor $A \hat{\otimes}_{A'}-$ to (\ref{eq:original}),  we obtain a Tor sequence
\begin{equation}
\label{eq:torry}
0 \to \mathrm{Tor}^{A'}_1(A,Z) \xrightarrow{\sigma} A \hat{\otimes}_{A'} T 
\xrightarrow{A\hat{\otimes}_{A'}\delta} A \hat{\otimes}_{A'} P^0 
\xrightarrow{A\hat{\otimes}_{A'}\epsilon} Z \to 0.
\end{equation}
Applying the functor $-\hat{\otimes}_{A'} Z$ to the exact sequence
$$0 \to J \to A' \to A\to 0,$$
we obtain a canonical isomorphism
\begin{equation}
\label{eq:iotadefmodule}
\iota:\mathrm{Tor}^{A'}_1(A,Z) \to J \hat{\otimes}_{A'} Z = J \hat{\otimes}_{A} Z
\end{equation}
since $A \hat{\otimes}_{A'} Z = Z$.  Combining (\ref{eq:torry}) and (\ref{eq:iotadefmodule})
gives an exact sequence
\begin{equation}
\label{eq:torry2}
0 \to J \hat{\otimes}_{A} Z \xrightarrow{\sigma\circ\iota^{-1}} A \hat{\otimes}_{A'} T 
\xrightarrow{A\hat{\otimes}_{A'}\delta} A \hat{\otimes}_{A'} P^0 
\xrightarrow{A\hat{\otimes}_{A'}\epsilon} Z \to 0.
\end{equation}
Let $\omega(Z,A')$ be the class of (\ref{eq:torry2}) in $\mathrm{Ext}^2_{B}(Z,J \hat{\otimes}_{A} Z)$.
Using the fact that $E_2^{p,q} = H_{II}^p(H_I^q(L^{\bullet,\bullet}))$ one can show
that $\omega(Z,A')$ is the image of 
$$\iota \in \mathrm{Hom}_B(\mathrm{Tor}^{A'}_1(A,Z), J \hat{\otimes}_{A} Z) = E_2^{0,1}$$
under the boundary map 
$$d_2^{0,1}:E_2^{0,1} \to E_2^{2,0}$$
associated to the spectral sequence (\ref{eq:spectralmodule}).  

We now sketch why $\omega(Z,A')$ is the obstruction to lifting $Z$
to a pseudocompact $B'$-module $Y$ which is (abstractly) free and finitely  generated
over $A'$ such that $A \hat{\otimes}_{A'} Y \cong Z$.  
If such  a lift $Y$ exists, one has an exact sequence of $B'$-modules 
\begin{equation}
\label{eq:xrated}
 0 \to X \to Y \to Z \to 0
 \end{equation}
in which $X$ is isomorphic to $JY = J \hat{\otimes}_{A} Z$.  The associated Tor sequence
\begin{equation}
\label{eq:torry3}
0 \to \mathrm{Tor}^{A'}_1(A,Z) \xrightarrow{f} A \hat{\otimes}_{A'} X \to A \hat{\otimes}_{A'} Y \xrightarrow{\upsilon} Z \to 0
\end{equation}
has the property that $\upsilon$ is an isomorphism, so $f$ is an isomorphism. Thus
(\ref{eq:torry3}) has trivial extension class.  By constructing a map from 
(\ref{eq:torry}) to (\ref{eq:torry3}) which is an identity on the leftmost and rightmost terms
we see $\omega(Z,A') = 0$.
Conversely, suppose that $\omega(Z,A') = 0$. Define $D$ to be the kernel
of the homomorphism $A \hat{\otimes}_{A'}\epsilon$ in (\ref{eq:torry}). By dimension shifting,
$\omega(Z,A') = 0$ implies that the exact sequence 
\begin{equation}
\label{eq:biteme}
0 \to J \hat{\otimes}_A Z \xrightarrow{\sigma\circ\iota^{-1}} A \hat{\otimes}_{A'} T 
\xrightarrow{A\hat{\otimes}_{A'}\delta} D=
\mathrm{Image}(A\hat{\otimes}_{A'}\delta) \to 0
\end{equation}
is split by a homomorphism
$\kappa: A \hat{\otimes}_{A'} T \to J \hat{\otimes}_A Z$ of pseudocompact $B$-modules. 
We now define $Y$ to be the pushout of $T\xrightarrow{\delta} P^0$ in
(\ref{eq:original}) and the composition $T\to A \hat{\otimes}_{A'} T \xrightarrow{\kappa} 
J \hat{\otimes}_A Z$.  One then has an exact sequence of the form (\ref{eq:xrated})
with $X = J \hat{\otimes}_A Z$.  On identifying $f$ in the resulting sequence (\ref{eq:torry3})
with $\kappa\circ \sigma = \iota$, one sees that $f$  is an isomorphisms.  Therefore
$\upsilon$ in (\ref{eq:torry3}) is an isomorphism, which shows $Y$ is a lift of $Z$. 

It follows from the sequence (\ref{eq:lowmodule}) of low degree terms that if 
there exists a lift of $Z$ over $A'$, i.e. if $\omega(Z,A')=0$, then the set of all local isomorphism classes of
lifts of $Z$ over $A'$ is in bijection with the full preimage of $\iota$ in 
$\mathrm{Ext}^1_{B'}(Z,J\hat{\otimes}_AZ)$ and is therefore a principal homogeneous space for 
$E_2^{1,0}=\mathrm{Ext}^1_B(Z,J\hat{\otimes}_AZ)$.

We now describe  the counterpart of the spectral sequence (\ref{eq:spectralmodule}) for a complex  
$Z^\bullet$ in place of $Z$.  
Assume Hypothesis \ref{hypo:obstruct}. The main point of assuming that $\HH^i(V^\bullet)=0$
unless $-p_0\le i\le -1$ is that this allows us to work in the abelian categories $C_0(B)$
and $C_0(B')$ of bounded above complexes that are concentrated in degrees $\le 0$.
Moreover, by insisting that $\HH^0(V^\bullet)$ is zero, we can make sure there exists an acyclic
complex of projective pseudocompact $B'$-modules $P^{0,\bullet}$ in $C_0(B')$ together with a 
morphism $\epsilon:P^{0,\bullet}\to Z^\bullet$ in $C_0(B')$ that is surjective on terms.
One can now generalize the spectral sequence 
(\ref{eq:spectralmodule}) by choosing 
a projective resolution $P^{\bullet,\bullet}$ of $Z^\bullet$ of projective
objects in $C_0(B')$ such that $P^{0,\bullet}$ has the nice properties above.
We then work with a  triple complex $M^{\bullet,\bullet,\bullet}$ which is a Cartan-Eilenberg resolution
of $A\hat{\otimes}_{A'}P^{\bullet,\bullet}$.  The double complex $L^{\bullet,\bullet}$ of $B$-modules
which leads to the spectral sequence we require is a partial total complex of the quadruple
complex $\mathrm{Hom}_B(M^{\bullet,\bullet,\bullet},J \hat{\otimes}_A Z^\bullet)$.  The
spectral sequence which results has the form
\begin{equation}
\label{eq:redundant}
E_2^{p,q}=\mathrm{Ext}^p_{D^-(B)}(\HH_{I}^{-q}(A\hat{\otimes}_{A'}P^{\bullet,\bullet}),
J\hat{\otimes}_AZ^\bullet)\;\Longrightarrow\;
\mathrm{Ext}^{p+q}_{D^-(B')}(Z^\bullet, J\hat{\otimes}_AZ^{\bullet})
\end{equation}
(see also (\ref{eq:spectral})).
As in the module case, we obtain an exact sequence of low degree terms, which looks slightly more
complicated than the sequence (\ref{eq:lowmodule}):
\begin{equation}
\label{eq:redundant2}
0\to E_2^{1,0}/W_2^{1,0} \to F_{II}^0\,\HH^1(\mathrm{Tot}(L^{\bullet,\bullet})) \to E_2^{0,1}/W_2^{0,1} 
\xrightarrow{\overline{d_2^{0,1}}} E_2^{2,0}
\end{equation}
(see also (\ref{eq:lowdegree})). Here $E_\infty^{1,0}=E_2^{1,0}/W_2^{1,0}$,
$E_\infty^{0,1}=\mathrm{Ker}(d_2^{0,1})/W_2^{0,1}$ and 
$F_{II}^0\,\HH^1(\mathrm{Tot}(L^{\bullet,\bullet}))$ is the second to last term in the second filtration
of 
$\HH^1(\mathrm{Tot}(L^{\bullet,\bullet}))=\mathrm{Ext}^1_{D^-(B')}(Z^\bullet,J\hat{\otimes}_AZ^\bullet)$.
The details of the set-up of the spectral sequence (\ref{eq:redundant}) and the sequence of
low degree terms (\ref{eq:redundant2}) for complexes $Z^\bullet$ are explained in 
\S \ref{ss:spectralseq}.

To define lifting obstructions, we follow the outlined construction in the module case given by 
equations (\ref{eq:original}) - (\ref{eq:torry2}). We assume as before that $P^{0,\bullet}$ is an acyclic complex of projective pseudocompact
$B'$-modules in $C_0(B')$. In particular, $\iota$ is an isomorphism in $C^-(B)$ and
our candidate for the lifting obstruction $\omega(Z^\bullet,A')$ is an element of
$\mathrm{Ext}^2_{D^-(B)}(Z^\bullet, J\hat{\otimes}_AZ^\bullet)$. Using the definition of
$L^{\bullet,\bullet}$ and the projective Cartan-Eilenberg resolution $M^{\bullet,\bullet,\bullet}$ of 
$A\hat{\otimes}_{A'}P^{\bullet,\bullet}$, we see, similarly to the module case, that 
$\omega(Z^\bullet,A')$ is the image of $\iota$ under the boundary map $d_2^{0,1}$
associated to the spectral sequence (\ref{eq:redundant}) (see Lemma \ref{lem:oyoyoy!}).

A  complication in the case of complexes compared to the module case is that
in the sequence of low degree terms (\ref{eq:redundant2}) the term 
$F_{II}^0=F_{II}^0 \,\HH^1(\mathrm{Tot}(L^{\bullet,\bullet}))$ is usually a proper subspace of
$\mathrm{Ext}^1_{B'}(Z^\bullet,J\hat{\otimes}_AZ)$. Therefore, we analyze in
\S \ref{ss:gabber} this subspace $F_{II}^0$. We use Gabber's ideas to see that $F_{II}^0$
consists precisely of those elements in $\mathrm{Ext}^1_{B'}(Z^\bullet,J\hat{\otimes}_AZ)$
which can be realized by short exact sequences in $D^-(B')$ of the form
$$\xi:\quad 0\to X^\bullet\to Y^\bullet\to Z^\bullet$$
where the terms of $X^\bullet$ are annihilated by $J$ and there exists an isomorphism
$h_\xi:X^\bullet\to J\hat{\otimes}AZ^\bullet$ in $D^-(B)$. A crucial step in showing this is 
to rewrite the elements of $F_{II}^0$ in terms of morphisms 
$\kappa\in\mathrm{Hom}_{D^-(B)}(A\hat{\otimes}_{A'}T^\bullet,J\hat{\otimes}_A Z^\bullet)$
(see Definition \ref{def:pushout} and Lemma \ref{lem:gabberconstruct}).
We then use the definition of $L^{\bullet,\bullet}$ and in particular the triple complex 
$M^{\bullet,\bullet,\bullet}$ to represent the class in 
$\mathrm{Ext}^1_{B'}(Z^\bullet,J\hat{\otimes}_AZ)$ given by $(\xi,h_\xi)$ explicitly as an
element in $L^{1,0}$, and hence as an element in $F_{II}^0$.
Finally we analyze the image of $E_\infty^{1,0}=E_2^{1,0}/W_2^{1,0}$ in $F_{II}^0$
in (\ref{eq:redundant2}) and describe the map $F_{II}^0\to E_2^{0,1}/W_2^{0,1}$ in 
(\ref{eq:redundant2}) to show that every element in $F_{II}^0$ can be represented by
a short exact sequence $\xi$ and an isomorphism $h_\xi$ as above. 
These steps are carried out in the proof of Lemma  \ref{lem:gabberfilter}.

The proof that $\omega(Z^\bullet,A')=0$ if and only if $Z^\bullet$ has a quasi-lift over $A'$
is then done in a very similar way to the module case (see Lemmas
\ref{lem:gabberlemma1} and  \ref{lem:gabberlift}).

Another  complication in the complex case is that the left most term in the
sequence (\ref{eq:redundant2}) is $E_\infty^{1,0}=E_2^{1,0}/W_2^{1,0}$ rather than
$E_2^{1,0}=\mathrm{Ext}^1_{D^-(B)}(Z^\bullet,J\hat{\otimes}_AZ^\bullet)$. 
As in the module case, we can directly use (\ref{eq:redundant2}) together with our
analysis of $F_{II}^0$ to show that if $\omega(Z^\bullet,A')=0$ then
the set of all local isomorphism classes of quasi-lifts of $Z^\bullet$ over $A'$ is a principal
homogeneous space for $E_\infty^{1,0}$. We then show that the
existence of a quasi-lift of $Z^\bullet$ over $A'$  implies that the spectral sequence
(\ref{eq:redundant}) partially degenerates. More precisely, we show that the inflation map 
$$\mathrm{Inf}_B^{B'}: \mathrm{Ext}^p_{D^-(B)}(Z^\bullet,J\hat{\otimes}_AZ^\bullet)\to
\mathrm{Ext}^p_{D^-(B')}(Z^\bullet,J\hat{\otimes}_AZ^\bullet)$$
is injective for all $p$ if $\omega(Z^\bullet,A')=0$. 
This is carried out in the proof of Lemma \ref{lem:gabberlift}.

%%%%%%%%%%%%%%%%%%%%%%%%%%%%%%%%%%%%%%%%%%%%%%%%%%%
%%A spectral sequence
%%%%%%%%%%%%%%%%%%%%%%%%%%%%%%%%%%%%%%%%%%%%%%%%%%%

\subsection{A spectral sequence}
\label{ss:spectralseq}

In this subsection we describe the spectral sequence we will use for the obstructions.
The definition of this spectral sequence follows (the dual of) Grothendieck's
construction in \cite[\S (11.7) of Chap. 0]{ega3}. 
The following remark describes certain subcategories of $C^-(B')$ and $C^-(B)$ which
play an important role in this construction.

\begin{rem}
\label{rem:projectivegrothendieck}
Suppose $\Lambda=B'$ or $B$. Let $C_0^-(\Lambda)$ be the full subcategory of
$C^-(\Lambda)$ whose objects are bounded above complexes $M^\bullet$ with $M^i=0$ for $i>0$.
Then $C_0^-(\Lambda)$ is an abelian category with enough projective objects.
More precisely, we have the following result which provides a slight correction
of \cite[Lemma 11.5.2.1]{ega3}, but which is proved in a similar fashion.

Let $\mathcal{P}$ be the set of all complexes $P^\bullet=(P^{-n})_{n\ge 0}$
 in $C_0^-(\Lambda)$ having the following 
properties: Every $P^{-n}$ is projective, $\BB^{-n}(P^\bullet)$ is a direct summand of $P^{-n}$ for
$n\ge 0$, and $\BB^{-n}(P^\bullet)=\ZZ^{-n}(P^\bullet)$ for $n\ge 1$. Then
\begin{itemize}
\item[(i)] $\mathcal{P}$ is the set of projective objects in $C_0^-(\Lambda)$,
and
\item[(ii)] every $M^\bullet$ in $C_0^-(\Lambda)$ is a homomorphic image of a complex in 
$P^\bullet\in\mathcal{P}$.
\end{itemize}

Note that $P^\bullet\in\mathcal{P}$ is not acyclic in general,
but that $\HH^{-n}(P^\bullet)=0$ for
$n\ge 1$ and $\HH^0(P^\bullet)$ is a projective pseudocompact $\Lambda$-module. 
\end{rem}

We will use a projective
resolution $P^{\bullet,\bullet}$ of $Z^\bullet$ in the category $C_0(B')$
of the following kind.

\begin{dfn}
\label{def:spectralseqA}
Choose a resolution of $Z^\bullet$ by projective objects in $C_0^-(B')$
\begin{equation}
\label{eq:doublep}
\cdots \to P^{-2,\bullet}\to P^{-1,\bullet}\to P^{0,\bullet}\xrightarrow{\epsilon} Z^\bullet\to 0
\end{equation}
such that $P^{-x,-y}=0$ unless $x\ge 0$ and $0\le y\le p_0$. 

Note that $P^{\bullet,\bullet}$ has commuting
differentials $d'_P$ and $d''_P$. 
We use the same convention as in \cite[\S (11.3) of Chap. 0]{ega3}
with respect to the differential of the total complex $\mathrm{Tot}(P^{\bullet,\bullet})$. Namely,  
$\mathrm{Tot}(P^{\bullet,\bullet})^{-n}=\bigoplus_{-x-y=-n}P^{-x,-y}$ and the differential
is given by $d\,a=d'_P\,a+(-1)^x\,d''_P\,a$ for $a\in P^{-x,-y}$. 

Define the map $\pi_P:
\mathrm{Tot}(P^{\bullet,\bullet})\to Z^\bullet$ by letting $\pi_P^{-n}:\mathrm{Tot}(P^{\bullet,\bullet})^{-n}
\to Z^{-n}$ be the composition of the natural projection $\mathrm{Tot}(P^{\bullet,\bullet})^{-n}\to P^{0,-n}$ 
with $\epsilon^{-n}: P^{0,-n}\to Z^{-n}$. Then $\pi_P$ defines a quasi-isomorphism in $C_0^-(B')$ 
that is surjective on terms.
\end{dfn}

Using the projective resolution $P^{\bullet,\bullet}$ of $Z^\bullet$ in $C_0(B')$, 
we can describe the spectral sequence as follows.

\begin{dfn}
\label{def:spectralseqA1}
Assume the notation of Definition $\ref{def:spectralseqA}$.
Taking the contravariant functor
$$\mathrm{Hom}_B(-,J\hat{\otimes}_AZ^\bullet) : \mathrm{PCMod}(B)\to C^-(B),$$
one shows similarly to \cite[\S (11.7) of Chap. 0]{ega3} that there is a convergent spectral sequence
\begin{equation}
\label{eq:spectral0}
\HH^{p}(\mathbf{R}\mathrm{Hom}^\bullet_B(\HH_{I}^{-q}(A\hat{\otimes}_{A'}P^{\bullet,\bullet}),
J\hat{\otimes}_AZ^\bullet))\;\Longrightarrow\; 
\HH^{p+q}(\mathbf{R}\mathrm{Hom}^\bullet_B(A\hat{\otimes}_{A'}P^{\bullet,\bullet},
J\hat{\otimes}_AZ^\bullet)).
\end{equation}
Here $\HH_{I}^{-q}(A\hat{\otimes}_{A'}P^{\bullet,\bullet})$ is the complex resulting from
taking the $-q^{\mathrm{th}}$ cohomology in the first direction of $A\hat{\otimes}_{A'}P^{\bullet,\bullet}$.
Using that
$\mathbf{R}\mathrm{Hom}^\bullet_B(A\hat{\otimes}_{A'}P^{\bullet,\bullet},
J\hat{\otimes}_AZ^\bullet)\cong \mathbf{R}\mathrm{Hom}^\bullet_{B'}(P^{\bullet,\bullet},
J\hat{\otimes}_AZ^\bullet)$,
the spectral sequence $(\ref{eq:spectral0})$ becomes
\begin{equation}
\label{eq:spectral}
E_2^{p,q}=\mathrm{Ext}^p_{D^-(B)}(\HH_{I}^{-q}(A\hat{\otimes}_{A'}P^{\bullet,\bullet}),
J\hat{\otimes}_AZ^\bullet)\;\Longrightarrow\;
\mathrm{Ext}^{p+q}_{D^-(B')}(Z^\bullet, J\hat{\otimes}_AZ^{\bullet}).
\end{equation}
Note that $\HH_{I}^{-q}(A\hat{\otimes}_{A'}P^{\bullet,\bullet})$ is the Tor complex
$\mathcal{H}^{-q}(A \hat{\otimes}_{A'}^{\LL} Z^\bullet)$ from $(\ref{eq:nightmare})$.
\end{dfn}

The proof of the convergence of the spectral sequence $(\ref{eq:spectral0})$ relies on the
existence of a projective Cartan-Eilenberg resolution $M^{\bullet,\bullet,\bullet}$
of $A\hat{\otimes}_{A'}P^{\bullet,\bullet}$. Moreover, the triple complex
$M^{\bullet,\bullet,\bullet}$ allows us to
realize the spectral sequence $(\ref{eq:spectral})$ as a spectral sequence of a double complex
$L^{\bullet,\bullet}$ relative to the second filtration of $\mathrm{Tot}(L^{\bullet,\bullet})$. 
We now give the definition of $M^{\bullet,\bullet,\bullet}$  and $L^{\bullet,\bullet}$.

\begin{dfn}
\label{def:spectralseqB}
Let $P^{\bullet,\bullet}$ be as in Definition $\ref{def:spectralseqA}$.
As described in \cite[\S (11.7) of Chap. 0]{ega3},
$A\hat{\otimes}_{A'}P^{\bullet,\bullet}$ admits a 
projective Cartan-Eilenberg resolution $M^{\bullet,\bullet,\bullet}=(M^{-x,-y,-z})$ where 
$x,z\ge 0$ and $0\le y\le p_0$.
This means that the terms  $M^{-x,-y,-z}$ are projective pseudocompact $B$-modules,
and for all $x$, $M^{-x,\bullet,\bullet}$ (resp. $\BB^{-x}_{I}(M^{\bullet,\bullet,\bullet})$,
resp. $\ZZ^{-x}_{I}(M^{\bullet,\bullet,\bullet})$, resp. $\HH^{-x}_{I}(M^{\bullet,\bullet,\bullet})$)
forms a projective resolution of $A\hat{\otimes}_{A'}(P^{-x,\bullet})$ (resp. 
$\BB^{-x}_{I}(A\hat{\otimes}_{A'}P^{\bullet,\bullet})$, resp. 
$\ZZ^{-x}_{I}(A\hat{\otimes}_{A'}P^{\bullet,\bullet})$, resp. 
$\HH^{-x}_{I}(A\hat{\otimes}_{A'}P^{\bullet,\bullet})$) in the abelian category $C_0^-(B)$.
In particular, $M^{-x,-y,\bullet}\to
A\hat{\otimes}_{A'}P^{-x,-y}\to 0$ is a projective resolution in the category $\mathrm{PCMod}(B)$ 
for all $x,y$.
The Cartan-Eilenberg property implies that we have for all $x,z$ split exact sequences of
complexes in $C_0^-(B)$
\begin{eqnarray}
\label{eq:carteil1}
0\to \BB^{-x}_{I}(M^{\bullet,\bullet,-z})\to \ZZ^{-x}_{I}(M^{\bullet,\bullet,-z}) \to 
\HH^{-x}_{I}(M^{\bullet,\bullet,-z})\to 0,\\
\label{eq:carteil2}
0\to \ZZ^{-x}_{I}(M^{\bullet,\bullet,-z}) \to M^{-x,\bullet,-z} \xrightarrow{d_{M,x}}
 \BB^{-x+1}_{I}(M^{\bullet,\bullet,-z})\to 0.
\end{eqnarray}

Since $M^{\bullet,\bullet,\bullet}$ has commuting
differentials $d_{M,x}$, $d_{M,y}$ and $d_{M,z}$, we use again the convention in 
\cite[\S (11.3) of Chap. 0]{ega3} with respect to the differential of the total complex 
$\mathrm{Tot}(M^{\bullet,\bullet,\bullet})$. 
Define the map $\pi_M:\mathrm{Tot}(M^{\bullet,\bullet,\bullet})\to
\mathrm{Tot}(A\hat{\otimes}_{A'}P^{\bullet,\bullet})$ by letting
$\pi_M^{-n}$ be the composition of the natural projection 
$\mathrm{Tot}(M^{\bullet,\bullet,\bullet})^{-n}\to \bigoplus_{-x-y=-n}M^{-x,-y,0}$ 
with the direct sum of the surjections $M^{-x,-y,0}\to  A\hat{\otimes}_{A'}P^{-x,-y}$.
Then $\pi_M$ defines a quasi-isomorphism in 
$C_0^-(B')$ that is surjective on terms.

Define a double complex $L^{\bullet,\bullet}$ of $B$-modules by
\begin{equation}
\label{eq:double}
L^{q,p}=\bigoplus_{-i+y+z=p} \mathrm{Hom}_B(M^{-q,-y,-z},J\hat{\otimes}_AZ^{-i}).
\end{equation}
Since $1\le i\le p_0$, $0\le y\le p_0$ 
and $z\ge 0$, it follows that for each integer $p$, there are only finitely many 
triples $(y,z,i)$ with $-i+y+z=p$. So we could also have used $\prod$ instead of
$\bigoplus$ in defining $L^{q,p}$.
Note that $L^{q,p}=0$ unless $q\ge 0$ and $p\ge -p_0$. In particular, for each integer $n$
there are only finitely many pairs $(q,p)$ with $q+p=n$ and $L^{q,p}\neq 0$. 
The differentials
$$d_I^{q,p} :L^{q,p}\to L^{q+1,p} \qquad\mbox{and}\qquad
d_{II}^{q,p}:L^{q,p}\to L^{q,p+1}$$
are described as follows:
\begin{eqnarray}
\label{eq:Ldifferential1}
d_I^{q,p}(g) &=& g\circ d_{M,x}^{-q-1,-y,-z},\\
\label{eq:Ldifferential2}
d_{II}^{q,p}(g) &=& g\circ d_{M,y}^{-q,-y-1,-z} + (-1)^y\,g\circ d_{M,z}^{-q,-y,-z-1} + (-1)^{p+1}\,
d_{J\hat{\otimes}_AZ}^{-i}\circ g
\end{eqnarray}
for $g\in \mathrm{Hom}_B(M^{-q,-y,-z},J\hat{\otimes}_AZ^{-i})$.
Since $d_I$ and $d_{II}$ commute, the total complex of $L^{\bullet,\bullet}$ whose 
$n^{\mathrm{th}}$ term is $\mathrm{Tot}(L^{\bullet,\bullet})^n=\bigoplus_{q+(-i+y+z)=n} 
\mathrm{Hom}_B(M^{-q,-y,-z},J\hat{\otimes}_AZ^{-i})$ has differential
$d$ with $d\,g = d_I^{q,p}(g)+(-1)^q\,d_{II}^{q,p}(g)$ for 
$g\in \mathrm{Hom}_B(M^{-q,-y,-z},J\hat{\otimes}_AZ^{-i})$.
Note that $\mathrm{Tot}(L^{\bullet,\bullet})$ is the total
Hom complex corresponding to the quadruple complex 
$\left(\mathrm{Hom}_B(M^{-q,-y,-z},J\hat{\otimes}_AZ^{-i})\right)_{q,y,z,i}$.
\end{dfn}

The following definition pertains to realizing the spectral sequence $(\ref{eq:spectral})$ as the
spectral sequence of $L^{\bullet,\bullet}$ relative to the second filtration of 
$\mathrm{Tot}(L^{\bullet,\bullet})$. This then leads to the sequences of low degree
terms corresponding to $(\ref{eq:spectral})$.

\begin{dfn}
\label{def:spectralseqB1}
Assume the notation of Definitions $\ref{def:spectralseqA}$ - $\ref{def:spectralseqB}$.
Let $\left( F_{II}^r(\mathrm{Tot}(L^{\bullet,\bullet}))\right)_{r\in\mathbb{Z}}$ 
be the filtration of the total complex $\mathrm{Tot}(L^{\bullet,\bullet})$ defined by
\begin{equation}
\label{eq:filtration}
F_{II}^r(\mathrm{Tot}(L^{\bullet,\bullet}))^n=\bigoplus_{q+p=n, p\ge r} L^{q,p} .
\end{equation}
Define $F_{II}^r\,\HH^n(\mathrm{Tot}(L^{\bullet,\bullet}))$ to be  the image in $\HH^n(
\mathrm{Tot}(L^{\bullet,\bullet}))$ of the $n$-cocycles in $F_{II}^r(\mathrm{Tot}(L^{\bullet,\bullet}))$,
i.e. of the elements in $F_{II}^r(\mathrm{Tot}(L^{\bullet,\bullet}))^n$ that are
in the kernel of the $n^{\mathrm{th}}$ differential of $\mathrm{Tot}(L^{\bullet,\bullet})$.

The spectral sequence $(\ref{eq:spectral})$ coincides with the spectral sequence of 
the double complex $L^{\bullet,\bullet}$
relative to the filtration $\left( F_{II}^r(\mathrm{Tot}(L^{\bullet,\bullet}))\right)_{r\in\mathbb{Z}}$ 
of $\mathrm{Tot}(L^{\bullet,\bullet})$ in $(\ref{eq:filtration})$.
In particular, $E_2^{p,q}=\HH_{II}^p(\HH_I^q(L^{\bullet,\bullet}))$ and
$\HH^{p+q}(\mathrm{Tot}(L^{\bullet,\bullet})) = 
\mathrm{Ext}^{p+q}_{D^-(B')}(Z^\bullet, J\hat{\otimes}_AZ^{\bullet})$.

We have a short exact sequence of low degree terms
\begin{equation}
\label{eq:lowdegree0}
0\to E_\infty^{1,0}\xrightarrow{\psi_{II}^0} F_{II}^0\,\HH^1(\mathrm{Tot}(L^{\bullet,\bullet})) 
\xrightarrow{\varphi_{II}^0} E_\infty^{0,1}\to 0.
\end{equation}
Here $E_\infty^{1,0}$ is the quotient of $E_2^{1,0}$ by the subgroup $W_2^{1,0}$ 
which is defined as the sum of the preimages in $E_2^{1,0}$ of the successive
images of $d_2^{-1,1}, d_3^{-2,2},\ldots$. Similarly $E_\infty^{0,1}$ is the quotient of 
$\mathrm{Ker}(d_2^{0,1})$ by the subgroup $W_2^{0,1}$ which is defined as 
the sum of the preimages in $\mathrm{Ker}(d_2^{0,1})$ of the successive images
of $d_2^{-2,2},d_3^{-3,3},\ldots$. Since $d_2^{0,1}:E_2^{0,1}\to E_2^{2,0}$ sends
$W_2^{0,1}$ identically to zero, the short exact sequence $(\ref{eq:lowdegree0})$ results in an exact 
sequence of low degree terms
\begin{equation}
\label{eq:lowdegree}
0\to E_\infty^{1,0} \xrightarrow{\psi_{II}^0} F_{II}^0\,\HH^1(\mathrm{Tot}(L^{\bullet,\bullet})) 
\xrightarrow{\tilde{\varphi}^0_{II}}  E_2^{0,1}/W_2^{0,1}
\xrightarrow{\overline{d_2^{0,1}}} E_2^{2,0}.
\end{equation}
\end{dfn}

%%%%%%%%%%%%%%%%%%%%%%%%%%%%%%%%%%%%%%%%%%%%%%%%%%%
%%Obstruction results
%%%%%%%%%%%%%%%%%%%%%%%%%%%%%%%%%%%%%%%%%%%%%%%%%%%

\subsection{Obstruction results}
\label{ss:obs}

In this subsection we list the main results concerning the obstruction
to lifting $(Z^\bullet,\zeta)$ to $A'$.
A key ingredient is a careful analysis of the exact
sequence of low degree terms in $(\ref{eq:lowdegree0})$.
The following definition is used to relate the term $F_{II}^0\,\HH^1(\mathrm{Tot}(L^{\bullet,\bullet}))$
in $(\ref{eq:lowdegree0})$ to extension classes
arising from short exact sequences of bounded above complexes of pseudocompact
$B'$-modules.

\begin{dfn}
\label{def:gabberclasses}
In $\mathrm{Ext}^1_{D^-(B')}(Z^\bullet, J\hat{\otimes}_{A}Z^\bullet)$ let $\tilde{F}_{II}^0$
be the subset of classes represented by short exact sequences in $C^-(B')$
\begin{equation}
\label{eq:gabberclass}
\xi:\qquad 0\to X^\bullet\xrightarrow{u_\xi} Y^\bullet \xrightarrow{v_\xi} Z^\bullet \to 0
\end{equation}
such that the terms of $X^\bullet$ are annihilated by $J$, and there is an isomorphism 
$h_\xi:X^\bullet \to J\hat{\otimes}_{A}Z^\bullet$ in $D^-(B)$.
Note that $h_\xi$ defines an isomorphism in $D^-(B')$. 
The triangle associated to the sequence
$\xi$ in $(\ref{eq:gabberclass})$ has the form
\begin{equation}
\label{eq:gabbertriangle}
X^\bullet\xrightarrow{u_\xi} Y^\bullet \xrightarrow{v_\xi} Z^\bullet\xrightarrow{w_\xi}X^\bullet[1]
\end{equation}
where $\eta_\xi=h_\xi[1]\circ w_\xi\in \mathrm{Hom}_{D^-(B')}(Z^\bullet, J\hat{\otimes}_{A}Z^\bullet[1])=
\mathrm{Ext}^1_{D^-(B')}(Z^\bullet,J\hat{\otimes}_{A}Z^\bullet)$ is the class
represented by $(\xi,h_\xi)$.
Applying the functor $A\hat{\otimes}_{A'}-$ to  $(\ref{eq:gabberclass})$ gives the long exact
Tor sequence in $C^-(B)$
\begin{equation}
\label{eq:torseq}
\cdots \to \mathrm{Tor}^{A'}_1(Y^\bullet,A)\to \mathrm{Tor}^{A'}_1(Z^\bullet, A)
\xrightarrow{f_\xi} X^\bullet \to A\hat{\otimes}_{A'}Y^\bullet \to Z^\bullet \to 0
\end{equation}
where $\mathrm{Tor}^{A'}_1(Z^\bullet, A)=\HH_I^{-1}(A\hat{\otimes}_{A'}P^{\bullet,\bullet})$
since $P^{\bullet,\bullet}$ in Definition \ref{def:spectralseqA} is a projective resolution
of $Z^\bullet$.
\end{dfn}

\begin{thm}
\label{thm:bigobstructionthm}
Assume Hypothesis $\ref{hypo:obstruct}$ and the notation introduced in Definitions 
$\ref{def:spectralseqA}$ - $\ref{def:gabberclasses}$.
The short exact sequence $(\ref{eq:lowdegree0})$ has the following properties.
\begin{enumerate}
\item[(i)] The group $F_{II}^0\,\HH^1(\mathrm{Tot}(L^{\bullet,\bullet}))$ equals the subset 
$\tilde{F}_{II}^0$ from Definition $\ref{def:gabberclasses}$.
\item[(ii)] The image of $E_\infty^{1,0}$ under $\psi_{II}^0$
in $F_{II}^0\,\HH^1(\mathrm{Tot}(L^{\bullet,\bullet}))=
\tilde{F}_{II}^0$ is equal to the subset of $\tilde{F}_{II}^0$ consisting of classes represented
by short exact sequences as in $(\ref{eq:gabberclass})$ where $Y^\bullet$ is in $C^-(B)$.
\item[(iii)] The map $\varphi_{II}^0:F_{II}^0\,\HH^1(\mathrm{Tot}(L^{\bullet,\bullet}))\to  E_\infty^{0,1}$
is defined in the following way.
Represent a class in $F_{II}^0\,\HH^1(\mathrm{Tot}(L^{\bullet,\bullet}))=\tilde{F}_{II}^0$ 
by $(\xi,h_\xi)$ as in Definition $\ref{def:gabberclasses}$. Let $f_\xi: \mathrm{Tor}^{A'}_1
(Z^\bullet, A)=\HH_I^{-1}(A\hat{\otimes}_{A'}P^{\bullet,\bullet})\to X^\bullet$ be as in
$(\ref{eq:torseq})$. Then $(\xi,h_\xi)$ is sent to the class of $h_\xi\circ f_\xi$ in $E_\infty^{0,1}$.
\end{enumerate}
\end{thm}

We obtain the following connection between the local isomorphism classes of quasi-lifts
of $(Z^\bullet,\zeta)$ over $A'$ and the classes in 
$F_{II}^0\,\HH^1(\mathrm{Tot}(L^{\bullet,\bullet}))=\tilde{F}_{II}^0$
defined by short exact sequences $\xi$ as in $(\ref{eq:gabberclass})$.

\begin{lemma}
\label{lem:extra}
Assume the hypotheses of Theorem $\ref{thm:bigobstructionthm}$.
If $(Z^\bullet,\zeta)$ has a quasi-lift over $A'$, then the local isomorphism class of every quasi-lift of 
$(Z^\bullet,\zeta)$ over $A'$ contains a quasi-lift $(Y^\bullet,\upsilon)$ such that $Y^\bullet$ occurs 
as the middle term of a short exact sequence $\xi$ as in $(\ref{eq:gabberclass})$.
\end{lemma}

The obstruction $\omega(Z^\bullet,A')$ to lifting $(Z^\bullet,\zeta)$ to $A'$ is defined in terms
of the following natural homomorphism in $C^-(B)$.

\begin{dfn}
\label{def:iota}
Let $\iota:\HH^{-1}_I(A\hat{\otimes}_{A'}P^{\bullet,\bullet})=
\mathrm{Tor}^{A'}_1(Z^\bullet,A)\to J\hat{\otimes}_{A}Z^\bullet$ 
be the natural homomorphism in $C^-(B)$ resulting from
tensoring the short exact sequence $0\to  J \to A'\to A\to 0$ with $Z^\bullet$ over $A'$.
Because the terms of $Z^\bullet$ are topologically flat $A$-modules by Hypothesis \ref{hypo:obstruct} ,
we get an exact sequence in $C^-(B)$
\begin{equation}
\label{eq:needthislater}
0\to \HH_I^{-1}(A\hat{\otimes}_{A'}P^{\bullet,\bullet})\xrightarrow{\iota}
J\hat{\otimes}_{A'}Z^\bullet\to A'\hat{\otimes}_{A'}Z^\bullet \xrightarrow{\cong}
A\hat{\otimes}_{A'}Z^\bullet \to 0.
\end{equation}
Hence $\iota$ is an isomorphism in $C^-(B)$.
\end{dfn}

\begin{thm}
\label{thm:obstructions}
Assuming the hypotheses of Theorem $\ref{thm:bigobstructionthm}$, 
let $\iota:\HH^{-1}_I(A\hat{\otimes}_{A'}P^{\bullet,\bullet})\to J\hat{\otimes}_{A}Z^\bullet$ 
be the isomorphism in $C^-(B)$ from Definition $\ref{def:iota}$.
If $[\iota]$ is the class of $\iota$ in $E_2^{0,1}/W_2^{0,1}$, let $\omega=\omega(Z^\bullet,A')$
be the class $\omega=\overline{d_2^{0,1}}([\iota])
=d_2^{0,1}(\iota)
\in E_2^{2,0}=\mathrm{Ext}^2_{D^-(B)}(Z^\bullet, J\hat{\otimes}_AZ^\bullet)$. 
\begin{enumerate}
\item[(i)] The class $\omega$ is zero if and only if there is a quasi-lift $(Y^\bullet,\upsilon)$ of 
$(Z^\bullet,\zeta)$ over $A'$.
\item[(ii)] If $\omega=0$, then $[\iota]\in E_\infty^{0,1}$ and
the set of all local isomorphism classes of quasi-lifts of $(Z^\bullet,\zeta)$ over 
$A'$ is in bijection with the full preimage of $[\iota]$ 
in $F_{II}^0\,\HH^1(\mathrm{Tot}(L^{\bullet,\bullet}))=\tilde{F}_{II}^0$ under $\varphi_{II}^0$.
In other words, the set of all local isomorphism classes of quasi-lifts of $(Z^\bullet,\zeta)$ over 
$A'$ is a principal homogeneous space for $E_\infty^{1,0}$.
\item[(iii)] If $\omega=0$, then $E_2^{p,0}=E_\infty^{p,0}$ for all $p$, i.e. the spectral sequence 
$(\ref{eq:spectral})$ partially degenerates. 
\end{enumerate}
\end{thm}

We will see in Remark $\ref{rem:notalwaysdegenerate}$ that if the lifting obstruction
$\omega(Z^\bullet,A')\neq 0$, then $E_\infty^{1,0}$ is a proper quotient of $E_2^{1,0}$ in general.

With respect to automorphisms of quasi-lifts, we get the following result.

\begin{lemma}
\label{lem:aut}
Assume the notation of Theorem $\ref{thm:obstructions}$, and suppose that
$\omega(Z^\bullet,A')=0$. Let $(Y^\bullet,\upsilon)$ be a quasi-lift of 
$(Z^\bullet,\zeta)$ over $A'$. 
Define $\mathrm{Aut}^0_{D^-(B')}(Y^\bullet)$ to be the group of automorphisms 
$\theta$ of $Y^\bullet$ in $D^-(B')$ for which
$\upsilon\circ (A\hat{\otimes}^{\LL}_{A'}\theta) = \upsilon$ in $D^-(B)$, i.e.
$A\hat{\otimes}^{\LL}_{A'}\theta$ is equal to the identity on $A\hat{\otimes}^{\LL}_{A'}Y^\bullet$
in $D^-(B)$. Then 
$$\mathrm{Aut}^0_{D^-(B')}(Y^\bullet)\cong \mathrm{Hom}_{D^-(B)}(
Z^\bullet,J\hat{\otimes}_A Z^\bullet)/\mathrm{Image}(
\mathrm{Ext}^{-1}_{D^-(B)}(Z^\bullet,Z^\bullet)).$$
Here $\mathrm{Image}(\mathrm{Ext}^{-1}_{D^-(B)}(Z^\bullet,Z^\bullet))$ is the image
of $\mathrm{Ext}^{-1}_{D^-(B)}(Z^\bullet,Z^\bullet)$ in $\mathrm{Hom}_{D^-(B)}(
Z^\bullet,J\hat{\otimes}_A Z^\bullet)$ under the map which is induced by the homomorphism
$A\hat{\otimes}^{\LL}_{A'}Y^\bullet[-1]\to J\hat{\otimes}^{\LL}_{A'}Y^\bullet$ in the triangle 
$A\hat{\otimes}^{\LL}_{A'}Y^\bullet[-1]\to  J\hat{\otimes}^{\LL}_{A'}Y^\bullet \to
 A'\hat{\otimes}^{\LL}_{A'}Y^\bullet\to A\hat{\otimes}^{\LL}_{A'}Y^\bullet$ in $D^-(B')$.
\end{lemma}

We obtain the following connection between the lifting obstruction $\omega(Z^\bullet,A')$ of
Theorem \ref{thm:obstructions} and the lifting obstruction $\omega_0(Z^\bullet,A')$
resulting from the naive approach described in \S \ref{ss:naive}.

\begin{prop}
\label{prop:compare}
Assume the notation of \S $\ref{ss:naive}$ and 
Theorem $\ref{thm:obstructions}$.
There exists an automorphism $u$ $($resp. $v$$)$ of $Z^\bullet$ 
$($resp. $J\hat{\otimes}_A Z^\bullet$$)$ in $D^-(B)$ such that
$\omega_0(Z^\bullet,A')=v[2]\circ \omega(Z^\bullet,A')\circ u$ in $D^-(B)$. 

Suppose $\omega(Z^\bullet,A')=0$. There exists an automorphism $u'$ $($resp. $v'$$)$ of $Z^\bullet$ 
$($resp. $J\hat{\otimes}_A Z^\bullet$$)$ in $D^-(B')$ with the following property:
Let $(Y_0^\bullet,\upsilon_0)$ and $({Y'}^\bullet,\upsilon')$ be two quasi-lifts of $(Z^\bullet,\zeta)$ 
over $A'$ whose local isomorphism classes correspond to $\eta_{\xi_0}$ and
$\eta_{\xi'}$, respectively, in $F_{II}^0\,\HH^1(\mathrm{Tot}(L^{\bullet,\bullet}))=\tilde{F}_{II}^0$
according to Lemma $\ref{lem:extra}$ and Theorem $\ref{thm:obstructions}(ii)$.
Then $\eta_{\xi'}-\eta_{\xi_0}=v'[1]\circ \varphi_{II}^0(\beta_{Y'})\circ u'$ in $D^-(B')$.
\end{prop}

The proofs of Theorems \ref{thm:bigobstructionthm}, \ref{thm:obstructions},
Lemma \ref{lem:aut} and Proposition \ref{prop:compare} are carried out in several sections.

%%%%%%%%%%%%%%%%%%%%%%%%%%%%%%%%%%%%%%%%%%%%%%%%%%%
%%Gabber's construction
%%%%%%%%%%%%%%%%%%%%%%%%%%%%%%%%%%%%%%%%%%%%%%%%%%%

\subsection{Gabber's construction}
\label{ss:gabber}

In this subsection we prove a result due to Gabber which is the key to
relating the term $F_{II}^0\,\HH^1(\mathrm{Tot}(L^{\bullet,\bullet}))$
from the sequence $(\ref{eq:lowdegree})$ to the set $\tilde{F}_{II}^0$ from Definition 
\ref{def:gabberclasses}.

\begin{dfn}
\label{def:gabber}
Assume Hypothesis $\ref{hypo:obstruct}$ and the notation introduced in Definitions
$\ref{def:spectralseqA}$ -  $\ref{def:gabberclasses}$.
We have a short exact sequence in $C^-(B')$
\begin{equation}
\label{eq:sequ1}
0\to T^\bullet \xrightarrow{\delta} P^{0,\bullet} \xrightarrow{\epsilon} Z^\bullet\to 0
\end{equation}
where $T^\bullet=\mathrm{Ker}(\epsilon)$ and $\delta$ is inclusion.
Recall that $P^{0,\bullet}$ is a projective object in $C_0^-(B')$. 
Since $Z^0=0$, we can, and will, assume 
that $P^{0,\bullet}$ is
an acyclic complex of projective pseudocompact $B'$-modules.
Tensoring $(\ref{eq:sequ1})$ with $A$ over $A'$ 
gives an exact sequence of complexes in $C^-(B)$
\begin{equation}
\label{eq:sequ2}
0\to \mathrm{Tor}^{A'}_1(Z^\bullet,A) \xrightarrow{\sigma} A\hat{\otimes}_{A'}T^\bullet
\xrightarrow{A\hat{\otimes}_{A'}\delta} A\hat{\otimes}_{A'}P^{0,\bullet} \xrightarrow{
A\hat{\otimes}_{A'}\epsilon} Z^\bullet \to 0
\end{equation}
where $\mathrm{Tor}^{A'}_1(Z^\bullet,A)=\HH_I^{-1}(A\hat{\otimes}_{A'}P^{\bullet,\bullet})$.
Write $(\ref{eq:sequ2})$ as the Yoneda composition of two short exact sequences in $C^-(B)$
\begin{eqnarray}
\label{eq:seqnow1}
&0\to D^\bullet\xrightarrow{\delta_D} A\hat{\otimes}_{A'}P^{0,\bullet} \xrightarrow{
A\hat{\otimes}_{A'}\epsilon} Z^\bullet \to 0,&
\\
\label{eq:seqnow2}
&0\to \HH_I^{-1}(A\hat{\otimes}_{A'}P^{\bullet,\bullet}) \xrightarrow{\sigma} 
A\hat{\otimes}_{A'}T^\bullet\xrightarrow{\tau} D^\bullet \to 0.&
\end{eqnarray}
Then the triangles in $D^-(B)$ associated to $(\ref{eq:seqnow1})$ and to 
$(\ref{eq:seqnow2})$ have the form
\begin{eqnarray}
\label{eq:alpha1}
& D^\bullet\xrightarrow{\delta_D} A\hat{\otimes}_{A'}P^{0,\bullet} \xrightarrow{
A\hat{\otimes}_{A'}\epsilon} Z^\bullet \xrightarrow{\alpha_1}D^\bullet[1],&
\\
\label{eq:alpha2}
& \HH_I^{-1}(A\hat{\otimes}_{A'}P^{\bullet,\bullet}) \xrightarrow{\sigma} 
A\hat{\otimes}_{A'}T^\bullet\xrightarrow{\tau} D^\bullet \xrightarrow{\alpha_2}
\HH_I^{-1}(A\hat{\otimes}_{A'}P^{\bullet,\bullet})[1].&
\end{eqnarray}
\end{dfn}

We first express the differential $d_2^{0,1}:E_2^{0,1}\to E_2^{2,0}$ in terms of 
the morphisms $\alpha_1$ and $\alpha_2$ 
in the triangles $(\ref{eq:alpha1})$ and $(\ref{eq:alpha2})$ in $D^-(B)$.

\begin{rem}
\label{rem:extraextra}
By $(\ref{eq:spectral})$ and Definition $\ref{def:spectralseqB1}$,
\begin{equation}
\label{eq:whenwillitend}
E_2^{p,q}=\mathrm{Ext}^p_{D^-(B)}(\HH_I^{-q}(A\hat{\otimes}_{A'}P^{\bullet,\bullet}),
J\hat{\otimes}_AZ^\bullet)=\HH_{II}^p(H_I^q(L^{\bullet,\bullet})).
\end{equation}
Thus the elements in 
$E_2^{p,q}$
are represented by 
elements $\beta\in L^{q,p}$ satisfying $d_I^{q,p}(\beta)=0$ and $d_{II}^{q,p}(\beta)\in
\mathrm{Image}(d_I^{q-1,p+1})$. 
It follows from $(\ref{eq:double})$ that
\begin{equation}
\label{eq:nice}
L^{q,p}=\bigoplus_j \mathrm{Hom}_B(\mathrm{Tot}(M^{-q,\bullet,\bullet})^{-j},
J\hat{\otimes}_AZ^{-j+p}),
\end{equation}
which
is equal to the $0^{\mathrm{th}}$ term in the total Hom complex
$\mathrm{Hom}_B^\bullet(\mathrm{Tot}(M^{-q,\bullet,\bullet}),
J\hat{\otimes}_AZ^\bullet[p])$. 
\end{rem}

\begin{lemma}
\label{lem:oyoyoy!}
Assume the notation of Definition $\ref{def:gabber}$ and Remark $\ref{rem:extraextra}$, 
and in particular the notation of
$(\ref{eq:alpha1})$, $(\ref{eq:alpha2})$ and $(\ref{eq:whenwillitend})$.
If
$$f\in E_2^{0,1}=
\mathrm{Hom}_{D^-(B)}(\HH_I^{-1}(A\hat{\otimes}_{A'}P^{\bullet,\bullet}),J\hat{\otimes}_AZ^\bullet),$$
then $d_2^{0,1}(f)=f[2]\circ\alpha_2[1]\circ\alpha_1\in
\mathrm{Hom}_{D^-(B)}(Z^\bullet,J\hat{\otimes}_AZ^\bullet[2])
=\mathrm{Ext}^2_{D^-(B)}(Z^\bullet,J\hat{\otimes}_AZ^\bullet)
=E_2^{2,0}$.
\end{lemma}

\begin{proof}
It follows from Remark $\ref{rem:extraextra}$ that
if $\beta_f\in L^{1,0}$ represents $f\in E_2^{0,1}$,
then there exists $\gamma_f\in L^{0,1}$ with $d_{II}^{1,0}(\beta_f)=d_I^{0,1}(\gamma_f)$.
Hence $d_2^{0,1}(f) \in E_2^{2,0}$ is represented by $d_{II}^{0,1}(\gamma_f)\in L^{0,2}$.
A calculation using $(\ref{eq:Ldifferential1})$ and $(\ref{eq:Ldifferential2})$ shows that
$d_{II}^{0,1}(\gamma_f)$ also represents $f[2]\circ\alpha_2[1]\circ\alpha_1\in E_2^{2,0}$.
In carrying out this calculation, it is useful to represent $\alpha_1$ explicitly in 
$(\ref{eq:alpha1})$ using a quasi-isomorphism between the mapping cone
of $\delta_D$ and $Z^\bullet$, and similarly for $\alpha_2$ in $(\ref{eq:alpha2})$.
\end{proof}

The next definition gives a connection between morphisms $\kappa$ in
$\mathrm{Hom}_{D^-(B)}(A\hat{\otimes}_{A'}T^\bullet,J\hat{\otimes}_A Z^\bullet)$ and
elements in $\tilde{F}_{II}^0$.
This is the key to relating $\tilde{F}_{II}^0$ to $F_{II}^0\,\HH^1(\mathrm{Tot}(L^{\bullet,\bullet}))$.

\begin{dfn}
\label{def:pushout}
Assume the notation of Definition $\ref{def:gabber}$, so that in particular,
$P^{0,\bullet}$  is an acyclic complex of projective pseudocompact 
$B'$-modules.
Suppose $\kappa:A\hat{\otimes}_{A'}T^\bullet\to J\hat{\otimes}_A Z^\bullet$
is a homomorphism in $D^-(B)$. Then $\kappa$ can be represented as
\begin{equation}
\label{eq:kappa}
\kappa = s^{-1}\circ\tilde{\kappa}
\end{equation} 
for suitable homomorphisms
$s:J\hat{\otimes}_AZ^\bullet\to X^\bullet$ and $\tilde{\kappa}:A\hat{\otimes}_{A'}T^\bullet\to X^\bullet$ 
in $C^-(B)$ such that $s$ is a quasi-isomorphism.
We obtain a pushout diagram in $C^-(B')$
\begin{equation}
\label{eq:pushout}
\xymatrix{
0\ar[r] & T^\bullet\ar[d]_{a_T}\ar[r]^{\delta} & P^{0,\bullet}\ar[r]^{\epsilon}\ar[dd]_\lambda 
& Z^\bullet\ar[r]\ar@{=}[dd] & 0\\
&A\hat{\otimes}_{A'}T^\bullet\ar[d]_{\tilde{\kappa}} \\
0\ar[r]& X^\bullet\ar[r]^{u_\xi}&Y^\bullet\ar[r]^{v_\xi}&Z^\bullet\ar[r]&0}
\end{equation}
where $a_T:T^\bullet\to A\hat{\otimes}_{A'}T^\bullet$ is the natural homomorphism in $C^-(B')$.
Let $\xi$ be the bottom row of $(\ref{eq:pushout})$ and let $h_\xi=s^{-1}$
in $D^-(B)$. Then $(\xi,h_\xi)$ represents a class $\eta_\xi\in\tilde{F}_{II}^0$ 
as in Definition $\ref{def:gabberclasses}$. 
Considering the triangles associated to the top and bottom rows of $(\ref{eq:pushout})$, we
obtain a commutative diagram in $D^-(B')$
\begin{equation}
\label{eq:triangle}
\xymatrix{
T^\bullet\ar[d]_{\tilde{\lambda}}\ar[r]^{\delta} & P^{0,\bullet}\ar[r]^{\epsilon}\ar[d]_\lambda 
& Z^\bullet\ar[r]^{\eta_T}\ar@{=}[d] &T^\bullet[1]\ar[d]^{\tilde{\lambda}[1]}\\
X^\bullet\ar[r]^{u_\xi}&Y^\bullet\ar[r]^{v_\xi}&Z^\bullet\ar[r]^{w_\xi}&X^\bullet[1]}
\end{equation}
where $\tilde{\lambda}=\tilde{\kappa}\circ a_T$.
Hence $\eta_\xi=h_\xi[1]\circ w_\xi=s^{-1}[1]\circ \tilde{\kappa}[1]\circ a_T[1]\circ \eta_T
=\kappa[1]\circ a_T[1]\circ\eta_T$. Thus the class $\eta_\xi\in \tilde{F}_{II}^0$
is independent of the choice of  the triple $(X^\bullet,s,\tilde{\kappa})$ 
used to represent $\kappa$, and we denote this class by $\eta_\kappa$.
In particular,
\begin{equation}
\label{eq:etakappa}
\eta_\kappa = \kappa[1]\circ a_T[1]\circ\eta_T.
\end{equation}
Since $P^{0,\bullet}$ is acyclic,
it follows that $\eta_T:Z^\bullet\to T^\bullet[1]$ is an isomorphism in $D^-(B')$.
Therefore it follows from $(\ref{eq:etakappa})$ that if $\kappa,\kappa'
\in \mathrm{Hom}_{D^-(B)}(A\hat{\otimes}_{A'}T^\bullet, J\hat{\otimes}_A
Z^\bullet)$, then $\eta_\kappa=\eta_{\kappa'}$ if and only if $\kappa\circ a_T=\kappa'\circ a_T$
in $D^-(B')$.
\end{dfn}

\begin{lemma}
\label{lem:gabberconstruct}
$\mathrm{[O.\;Gabber]}$
Assume the notation of Definition $\ref{def:gabber}$, so that in particular,
$P^{0,\bullet}$ is an acyclic complex of projective
pseudocompact $B'$-modules.
\begin{enumerate}
\item[(i)] Let $(\xi,h_\xi)$ represent a class $\eta_\xi$ in $\tilde{F}_{II}^0$ 
as in Definition $\ref{def:gabberclasses}$, and let $f_\xi: \mathrm{Tor}^{A'}_1
(Z^\bullet, A)=\HH_I^{-1}(A\hat{\otimes}_{A'}P^{\bullet,\bullet})\to X^\bullet$ be as in
$(\ref{eq:torseq})$. Then $d_2^{0,1}(h_\xi\circ f_\xi)=0$. 
Moreover, there exists 
$\kappa_\xi\in \mathrm{Hom}_{D^-(B)}(A\hat{\otimes}_{A'}T^\bullet, J\hat{\otimes}_AZ^\bullet)$ 
such that $h_\xi\circ f_\xi=\kappa_\xi\circ\sigma$ and $\eta_\xi=\eta_{\kappa_\xi}$,
where $\sigma$ is as in $(\ref{eq:seqnow2})$ and $\eta_{\kappa_\xi}$ is the class in $\tilde{F}_{II}^0$
defined by $\kappa_\xi$ as in Definition $\ref{def:pushout}$.
\item[(ii)] Conversely, suppose $f\in E_2^{0,1}=\mathrm{Hom}_{D^-(B)}(\HH_I^{-1}(A\hat{\otimes}_{A'}
P^{\bullet,\bullet}),J\hat{\otimes}_AZ^\bullet)$ satisfies $d_2^{0,1}(f)=0$.
Then there exists $\kappa\in \mathrm{Hom}_{D^-(B)}(A\hat{\otimes}_{A'}T^\bullet, J\hat{\otimes}_A
Z^\bullet)$ such that $\kappa\circ\sigma=f$. Moreover, if 
$\kappa'\in \mathrm{Hom}_{D^-(B)}(A\hat{\otimes}_{A'}T^\bullet, J\hat{\otimes}_A
Z^\bullet)$ also satisfies $\kappa'\circ\sigma=f$, then there exists
$\alpha\in \mathrm{Hom}_{D^-(B)}(D^\bullet,J\hat{\otimes}_AZ^\bullet)\cong \mathrm{Ext}^1_{D^-(B)}(
Z^\bullet,J\hat{\otimes}_AZ^\bullet)=E_2^{1,0}$ with $\kappa-\kappa'=\alpha\circ\tau$.
Let $\eta_\kappa\in \tilde{F}_{II}^0$ be the class defined by $\kappa$ as in Definition 
$\ref{def:pushout}$ 
and let $(\xi,h_\xi)$ be a representative.
Then the corresponding morphism 
$f_\xi: \HH_I^{-1}(A\hat{\otimes}_{A'}P^{\bullet,\bullet})\to X^\bullet$ from
$(\ref{eq:torseq})$ satisfies $h_\xi\circ f_\xi=f$.
\item[(iii)] Let $\tilde{F}_{II}^1$ be the subset of $\tilde{F}_{II}^0$ consisting of classes represented
by short exact sequences as in $(\ref{eq:gabberclass})$ where $Y^\bullet$ is in $C^-(B)$.
Then $\tilde{F}_{II}^1$ is equal to the set of all classes $\eta_{\kappa_\alpha}$ in $\tilde{F}_{II}^0$
defined by $\kappa_\alpha=\alpha\circ\tau$ as in Definition $\ref{def:pushout}$ as $\alpha$ varies
over all elements in $\mathrm{Hom}_{D^-(B)}(D^\bullet,J\hat{\otimes}_AZ^\bullet)\cong E_2^{1,0}$.
\end{enumerate}
\end{lemma}

\begin{proof}
Since $P^{0,\bullet}$ is acyclic, the morphism $\alpha_1:Z^\bullet\to D^\bullet[1]$ from 
$(\ref{eq:alpha1})$ is an isomorphism in $D^-(B)$. Thus 
$\mathrm{Hom}_{D^-(B)}(D^\bullet,J\hat{\otimes}_AZ^\bullet)\cong \mathrm{Ext}^1_{D^-(B)}(
Z^\bullet,J\hat{\otimes}_AZ^\bullet)=E_2^{1,0}$.
Moreover, using Lemma \ref{lem:oyoyoy!}, we have
that $d_2^{0,1}(f)=0$ if and only if $f[1]\circ\alpha_2=0$ in $D^-(B)$.

For part (i), let $(\xi,h_\xi)$ be as in Definition \ref{def:gabberclasses},
where $\xi:\;0\to X^\bullet\xrightarrow{u_\xi}Y^\bullet\xrightarrow{v_\xi}
Z^\bullet\to 0$.
Since $P^{0,\bullet}$ is a projective object in $C^-(B')$,
there exists a commutative diagram in $C^-(B')$ of the form
\begin{equation}
\label{eq:diagg1}
\xymatrix{
0\ar[r]& T^\bullet \ar[r]^{\delta}\ar[d]_{\tilde{\lambda}} & P^{0,\bullet}\ar[r]^{\epsilon}
\ar[d]_{\lambda} & Z^\bullet \ar[r]\ar@{=}[d]&0\\
0\ar[r]&X^\bullet \ar[r]^{u_\xi}&Y^\bullet\ar[r]^{v_\xi}&Z^\bullet\ar[r]& 0.}
\end{equation}
Because $J$ annihilates the terms of $X^\bullet$, $\tilde{\lambda}$
factors as $\tilde{\lambda}=(A\hat{\otimes}_{A'}\tilde{\lambda})\circ a_T$. Hence
$\xi$ is the bottom row of a pushout diagram as in $(\ref{eq:pushout})$ with
$\tilde{\kappa}=A\hat{\otimes}_{A'}\tilde{\lambda}$. 
Letting  $\kappa_\xi=h_\xi\circ(A\hat{\otimes}_{A'}\tilde{\lambda})$ gives $\eta_\xi=\eta_{\kappa_\xi}$ 
by Definition $\ref{def:pushout}$.
Tensoring $(\ref{eq:diagg1})$ with $A$ over $A'$ and using $(\ref{eq:sequ2})$ shows that
$(A\hat{\otimes}_{A'}\tilde{\lambda})\circ\sigma=f_\xi$.
Since $\alpha_2$ and $\sigma[1]$ are 
consecutive maps in the triangle obtained by shifting $(\ref{eq:alpha2})$, 
this implies that $f_\xi[1]\circ\alpha_2=0$ in $D^-(B)$, and hence 
$d_2^{0,1}(h_\xi\circ f_\xi)=0$. Moreover, $h_\xi\circ f_\xi=\kappa_\xi\circ\sigma$.

For part (ii), assume $d_2^{0,1}(f)=0$. 
Applying the functor $\mathrm{Hom}_{D^-(B)}(-,J\hat{\otimes}_AZ^\bullet)$ to
the triangle $(\ref{eq:alpha2})$, we obtain a long exact $\mathrm{Hom}$ sequence.
By the first paragraph of the proof, $f\circ\alpha_2[-1]=0$, which 
shows that there exists $\kappa\in \mathrm{Hom}_{D^-(B)}
(A\hat{\otimes}_{A'}T^\bullet,J\hat{\otimes}_AZ^\bullet)$ with $\kappa\circ\sigma = f$.
Let $\eta_\kappa\in \tilde{F}_{II}^0$ be the class defined by $\kappa$ as in Definition 
$\ref{def:pushout}$
and let $(\xi,h_\xi)$ be a representative.
In particular, $\kappa=h_\xi\circ\tilde{\kappa}$ where
$\tilde{\kappa}$ is as in $(\ref{eq:pushout})$ and $\xi$ is the bottom row of 
$(\ref{eq:pushout})$. Tensoring
$(\ref{eq:pushout})$ with $A$ over $A'$ and using $(\ref{eq:sequ2})$ shows that
$\tilde{\kappa}\circ\sigma=f_\xi$. This implies that
$h_\xi\circ f_\xi=\kappa\circ\sigma=f$.

For part (iii), let first $\alpha\in \mathrm{Hom}_{D^-(B)}(D^\bullet,J\hat{\otimes}_AZ^\bullet)$ and
let $\kappa=\alpha\circ\tau$. Following the construction of the class 
$\eta_\kappa\in \tilde{F}_{II}^0$ in Definition $\ref{def:pushout}$ which 
has representative $(\xi,h_\xi)$, we see that we can choose 
$\tilde{\kappa}$ in $(\ref{eq:kappa})$ and in 
$(\ref{eq:pushout})$ to be of the form $\tilde{\kappa}=\tilde{\mu}\circ\tau$ for a suitable 
$\tilde{\mu}:D^\bullet\to X^\bullet$ in $C^-(B)$. Using the definitions of $\delta_D$ and $\tau$ in
$(\ref{eq:seqnow1})$ and $(\ref{eq:seqnow2})$, it follows that
$\xi$ is the bottom row of a pushout diagram in $C^-(B)$
\begin{equation}
\label{eq:pushoutalpha}
\xymatrix{
0\ar[r] & D^\bullet\ar[d]_{\tilde{\mu}}\ar[r]^(.4){\delta_D} & A\hat{\otimes}_{A'}P^{0,\bullet}\ar[r]^(.65){
A\hat{\otimes}_{A'}\epsilon}\ar[d]_{\mu} & Z^\bullet\ar[r]\ar@{=}[d] & 0\\
0\ar[r]& X^\bullet\ar[r]^{u_\xi}&Y^\bullet\ar[r]^{v_\xi}&Z^\bullet\ar[r]&0}
\end{equation}
where the first row is given by $(\ref{eq:seqnow1})$.
This implies that $\eta_\kappa=\eta_\xi$ lies in 
$\tilde{F}_{II}^1$.
To prove the converse direction, one takes the representative $(\xi,h_\xi)$ of a class in $\tilde{F}_{II}^1$
and uses that $A\hat{\otimes}_{A'}P^{0,\bullet}$  is a projective object in $C^-(B)$ to realize
$\xi$ as the bottom row of a diagram  as in $(\ref{eq:pushoutalpha})$.
Letting $\tilde{\kappa}=\tilde{\mu}\circ\tau$, it follows that $\xi$ is also the bottom row 
of a pushout diagram as in $(\ref{eq:pushout})$. Define $\kappa=h_\xi\circ\tilde{\kappa}$ and
$\alpha=h_\xi\circ\tilde{\mu}\in\mathrm{Hom}_{D^-(B)}(D^\bullet,J\hat{\otimes}_AZ^\bullet)$.
Then $\kappa=\alpha\circ\tau$ and $\eta_\xi=\eta_{\kappa}$.
\end{proof}

%%%%%%%%%%%%%%%%%%%%%%%%%%%%%%%%%%%%%%%%%%%%%%%%%%%
%%Proof of Theorem $\ref{thm:bigobstructionthm}$
%%%%%%%%%%%%%%%%%%%%%%%%%%%%%%%%%%%%%%%%%%%%%%%%%%%

\subsection{Proof of Theorem $\ref{thm:bigobstructionthm}$}
\label{ss:bigobsthm}

In this subsection we prove Theorem $\ref{thm:bigobstructionthm}$ by proving 
Lemma \ref{lem:gabberfilter} given below. We use the following notation.

\begin{dfn}
Suppose $\Lambda=B'$ or $B$, and $M_1^\bullet$ and $M_2^\bullet$ are complexes in
$C^-(\Lambda)$. We say a homomorphism $f\in\mathrm{Hom}_{D^-(\Lambda)}(M_1^\bullet,
M_2^\bullet)$ is represented  by a homomorphism $f':{M_1'}^\bullet\to {M_2'}^\bullet$ in $C^-(\Lambda)$
(resp. in $D^-(\Lambda)$) if there exist isomorphisms $s_i:M_i^\bullet\to {M_i'}^\bullet$
in $D^-(\Lambda)$ for $i=1,2$ with $f=s_2^{-1}\circ f'\circ s_1$ in $D^-(\Lambda)$.
\end{dfn}

\begin{lemma}
\label{lem:gabberfilter}
Assume the notation of Definition $\ref{def:gabber}$, so that in particular, 
$P^{0,\bullet}$ is an acyclic complex of projective
pseudocompact $B'$-modules.
Let $(\xi,h_\xi)$ represent a class $\eta_\xi$ in $\tilde{F}_{II}^0$ 
as in Definition $\ref{def:gabberclasses}$, where
$\xi:\;0\to X^\bullet\xrightarrow{u_\xi}Y^\bullet\xrightarrow{v_\xi}Z^\bullet\to 0$.
Let $w_\xi\in \mathrm{Hom}_{D^-(B')}(Z^\bullet,
X^\bullet[1])$ be as in $(\ref{eq:gabbertriangle})$, and let $f_\xi: 
\HH_I^{-1}(A\hat{\otimes}_{A'}P^{\bullet,\bullet})\to X^\bullet$ be the connecting homomorphism
as in $(\ref{eq:torseq})$. 
\begin{enumerate}
\item[(i)] The class $\eta_\xi=h_\xi[1]\circ w_\xi$ lies in $F_{II}^0\,\HH^1(\mathrm{Tot}(L^{\bullet,
\bullet}))$. More precisely, $\eta_\xi$ defines an element $\beta_\xi\in L^{1,0}$ which lies in the
kernel of the first differential of $\mathrm{Tot}(L^{\bullet,\bullet})$.  This identifies $\tilde{F}_{II}^0$
with a subset of $F_{II}^0\,\HH^1(\mathrm{Tot}(L^{\bullet,\bullet}))$.
\item[(ii)] The map $\varphi_{II}^0:F_{II}^0\,\HH^1(\mathrm{Tot}(L^{\bullet,\bullet}))\to E_\infty^{0,1}$ in 
$(\ref{eq:lowdegree0})$ sends $\eta_\xi=h_\xi[1]\circ w_\xi$ to the class of $h_\xi\circ f_\xi$ in $E_\infty^{0,1}$.
This gives a surjection $\tilde{F}_{II}^0 \to E_\infty^{0,1}$. 
\item[(iii)] The image of $E_\infty^{1,0}$ in $F_{II}^0\,\HH^1(\mathrm{Tot}(L^{\bullet,\bullet}))$
under $\psi_{II}^0$ is equal to the subset $\tilde{F}_{II}^1$ of $\tilde{F}_{II}^0$ consisting of classes
represented by short exact sequences as in $(\ref{eq:gabberclass})$ where $Y^\bullet$ is in $C^-(B)$.
\item[(iv)] Fix an element $f\in \mathrm{Ker}(d_2^{0,1}:E_2^{0,1} \to E_2^{2,0})$ as in Lemma
$\ref{lem:gabberconstruct}(ii)$.  Let $\kappa$ vary over all choices of elements of 
$\mathrm{Hom}_{D^-(B)}(A\hat{\otimes}_{A'}T^\bullet, J\hat{\otimes}_AZ^\bullet)$
for which $\kappa\circ\sigma=f$.  Then the classes $\eta_\kappa$ in $\tilde{F}_{II}^0$,
as defined in Definition $\ref{def:pushout}$,
form a coset of $\psi_{II}^0(E_\infty^{1,0})$ in $F_{II}^0\,\HH^1(\mathrm{Tot}(L^{\bullet,\bullet}))$.  
\end{enumerate}
In particular,
$\tilde{F}_{II}^0 = F_{II}^0\,\HH^1(\mathrm{Tot}(L^{\bullet,\bullet}))$.
\end{lemma}

\begin{proof}
Let $\kappa_\xi\in\mathrm{Hom}_{D^-(B)}(A\hat{\otimes}_{A'}T^\bullet,J\hat{\otimes}_AZ^\bullet)$ 
be as in Lemma \ref{lem:gabberconstruct}(i). It follows from $(\ref{eq:etakappa})$
that we can write $\eta_\xi$ as
\begin{equation}
\label{eq:ohbabyo1}
\eta_\xi=\kappa_\xi[1]\circ a_T[1]\circ \eta_T
\end{equation}
where $\eta_T\in\mathrm{Hom}_{D^-(B')}(Z^\bullet,T^\bullet[1])$ is as in $(\ref{eq:triangle})$
and $a_T:T^\bullet\to A\hat{\otimes}_{A'}T^\bullet$ is the natural homomorphism in $C^-(B')$.

To prove part (i), one uses the projective Cartan-Eilenberg resolution $M^{\bullet,\bullet,\bullet}$
of $A\hat{\otimes}_{A'}P^{\bullet,\bullet}$ from Definition \ref{def:spectralseqB} to represent
$\kappa_\xi$ by the homotopy class in $K^-(B)$ of a homomorphism
\begin{equation}
\label{eq:kappaduo}
\kappa_{\xi,\sharp}: \frac{\mathrm{Tot}(M^{-1,\bullet,\bullet})}{
\mathrm{Tot}(\BB_I^{-1}(M^{\bullet,\bullet,\bullet}))}\to J\hat{\otimes}_AZ^\bullet
\end{equation}
in $C^-(B)$. To find $\kappa_{\xi,\sharp}$, one 
first identifies $T^\bullet$ with $\BB_I^0(P^{\bullet,\bullet})$ by $(\ref{eq:sequ1})$
and then shows that there are quasi-isomorphisms
\begin{equation}
\label{eq:qisuno}
\frac{\mathrm{Tot}(M^{-1,\bullet,\bullet})}{\mathrm{Tot}(\BB_I^{-1}(M^{\bullet,\bullet,\bullet}))}
\xrightarrow{\overline{\pi}_M^{\,-1,\bullet}} \frac{A\hat{\otimes}_{A'}P^{-1,\bullet}}{
\BB_I^{-1}(A\hat{\otimes}_{A'}P^{\bullet,\bullet})}
\xrightarrow{\overline{A\hat{\otimes}_{A'}d_P'}}A\hat{\otimes}_{A'}\BB_I^0(P^{\bullet,\bullet})
=A\hat{\otimes}_{A'}T^\bullet
\end{equation}
in $C^-(B)$,
where  $\overline{\pi}_M^{\,-1,\bullet}$ is induced by the quasi-isomorphism
$\pi_M$ from Definition \ref{def:spectralseqB}.
Using $(\ref{eq:carteil1})$ and $(\ref{eq:carteil2})$, it follows that
$\mathrm{Tot}(M^{-1,\bullet,\bullet})\,/\,\mathrm{Tot}(\BB_I^{-1}(M^{\bullet,\bullet,\bullet}))$
is a bounded above complex of projective pseudocompact $B$-modules.
Hence the composition of  $\kappa_\xi$ with the quasi-isomorphisms in $(\ref{eq:qisuno})$
represents $\kappa_\xi$ in $D^-(B)$ and is given by the homotopy class in $K^-(B)$ of a 
homomorphism $\kappa_{\xi,\sharp}$ as in $(\ref{eq:kappaduo})$.

Let $\pi_{\BB_I^{-1}}:\mathrm{Tot}(M^{-1,\bullet,\bullet})\to \mathrm{Tot}(M^{-1,\bullet,\bullet})\,/\,
\mathrm{Tot}(\BB_I^{-1}(M^{\bullet,\bullet,\bullet}))$ be the natural projection in $C^-(B)$ and define
$\beta_{\xi,j}\in \mathrm{Hom}_B(\mathrm{Tot}(M^{-1,\bullet,\bullet})^{-j},
J\hat{\otimes}_AZ^{-j})$ by
\begin{equation}
\label{eq:beta}
\beta_{\xi,j}=\kappa_{\xi,\sharp}^{-j}\circ \pi_{\BB_I^{-1}}^{-j}
\end{equation}
for all $j$.
By $(\ref{eq:nice})$, $\beta_\xi=(\beta_{\xi,j})$ defines an element in $L^{1,0}$.
It follows from the construction that $d_I^{1,0}(\beta_\xi)=0=d_{II}^{1,0}(\beta_\xi)$. 
By considering the effect of making a different choice of $\kappa_\xi$ in $(\ref{eq:ohbabyo1})$,
one sees that the class $[\beta_\xi]$ in $F_{II}^0\,\HH^1(\mathrm{Tot}(L^{\bullet,\bullet}))$
only depends on $\eta_\xi\in\tilde{F}_{II}^0$. Hence the map $\eta_\xi\mapsto [\beta_\xi]$ shows that
$\tilde{F}_{II}^0\subseteq F_{II}^0\,\HH^1(\mathrm{Tot}(L^{\bullet,\bullet}))$.

Part (ii) is proved by considering the restriction of the homomorphism $\kappa_{\xi,\sharp}$ in
$C^-(B)$ from $(\ref{eq:kappaduo})$ to $\mathrm{Tot}(\HH_I^{-1}(M^{\bullet,\bullet,\bullet}))$.
By Lemma \ref{lem:gabberconstruct}(i), we have $h_\xi\circ f_\xi=\kappa_\xi\circ\sigma$, where 
$\sigma: \HH_I^{-1}(A\hat{\otimes}_{A'}P^{\bullet,\bullet})\to A\hat{\otimes}_{A'}
\BB_I^0(P^{\bullet,\bullet})=A\hat{\otimes}_{A'}T^\bullet$ is the homomorphism 
from $(\ref{eq:seqnow2})$. Using the projective Cartan-Eilenberg resolution
$M^{\bullet,\bullet,\bullet}$,
one sees that $\sigma$ is represented in $D^-(B)$ by 
the restriction of the composition of the quasi-isomorphisms in $(\ref{eq:qisuno})$ 
to $\mathrm{Tot}(\HH_I^{-1}(M^{\bullet,\bullet,\bullet}))$.
Since  $\kappa_{\xi,\sharp}$  represents the
composition of  $\kappa_\xi$ with the quasi-isomorphisms in $(\ref{eq:qisuno})$, it follows that
the restriction of $\kappa_{\xi,\sharp}$ to $\mathrm{Tot}(\HH_I^{-1}(M^{\bullet,\bullet,\bullet}))$ 
represents $\kappa_\xi\circ \sigma=h_\xi\circ f_\xi$.
This implies that $\varphi_{II}^0$ sends 
$\eta_\xi$
to the class of $h_\xi\circ f_\xi$ in 
$E_\infty^{0,1}$. It follows from Lemma \ref{lem:gabberconstruct}(ii) that the restriction of
$\varphi_{II}^0$ to $\tilde{F}_{II}^0$ gives a surjection $\tilde{F}_{II}^0 \to E_\infty^{0,1}$.

To prove part (iii), one relates
the elements of $\tilde{F}_{II}^1$ and of $F_{II}^1\,\HH^1(\mathrm{Tot}(L^{\bullet,\bullet}))$
to elements in $L^{0,0}$, using the differentials in $(\ref{eq:Ldifferential1})$ and 
$(\ref{eq:Ldifferential2})$.
Let first $(\xi,h_\xi)$ represent a class in $\tilde{F}_{II}^1$.
By Lemma \ref{lem:gabberconstruct}(iii), there exists a morphism
$\alpha_\xi\in\mathrm{Hom}_{D^-(B)}(D^\bullet,J\hat{\otimes}_{A'}Z^\bullet)$ such that
$\eta_\xi=\eta_{\kappa_\xi}$ for $\kappa_\xi=\alpha_\xi\circ\tau$. 
By analyzing
the construction of $\beta_\xi=(\beta_{\xi,j})\in L^{1,0}$ in $(\ref{eq:beta})$ 
for $\kappa_\xi=\alpha_\xi\circ\tau$, one shows that there exists 
$\gamma=(\gamma_j) \in L^{0,0}$ such that
\begin{equation}
\label{eq:sooodumb}
[\beta_\xi]=[-d_{II}^{0,0}(\gamma)]
\end{equation}
in $F_{II}^0\,\HH^1(\mathrm{Tot}(L^{\bullet,\bullet}))$.
To construct $\gamma=(\gamma_j)$, one represents $\alpha_\xi$  by a homomorphism of complexes 
\begin{equation}
\label{eq:kappaduo1}
\alpha_{\xi,\sharp}: \mathrm{Tot}(\BB_I^0(M^{\bullet,\bullet,\bullet}))\to J\hat{\otimes}_AZ^\bullet
\end{equation}
in $C^-(B)$ and defines $\gamma_j \in \mathrm{Hom}_B(\mathrm{Tot}(
M^{0,\bullet,\bullet})^{-j},J\hat{\otimes}_AZ^{-j})$ by
\begin{equation}
\label{eq:L00}
\gamma_j=\alpha_{\xi,\sharp}^{-j}\circ \mathrm{proj}_{0,j},
\end{equation}
where $\mathrm{proj}_{0,j}:\mathrm{Tot}(M^{0,\bullet,\bullet})^{-j}\to
\mathrm{Tot}(\BB_I^0(M^{\bullet,\bullet,\bullet}))^{-j}$
is induced by the projections
$M^{0,-y,-z}\to \BB_I^0(M^{\bullet,-y,-z})$ for all $y,z$ with $y+z=j$
coming from the split exact sequences $(\ref{eq:carteil1})$ and $(\ref{eq:carteil2})$.
Using $(\ref{eq:Ldifferential1})$, one checks that 
$[\beta_\xi]=[d_I^{0,0}(\gamma)]$ in $F_{II}^0\,\HH^1(\mathrm{Tot}(L^{\bullet,\bullet}))$, 
which implies (\ref{eq:sooodumb}) because $[d_{\mathrm{Tot}(L)}^0(\gamma)]=0$.
Hence $[\beta_\xi]$ is equal to an element in $F_{II}^1\,\HH^1(\mathrm{Tot}(L^{\bullet,\bullet}))
=\psi_{II}^0(E_\infty^{1,0})$.

Conversely, suppose $\beta=(\beta_j)\in L^{0,1}
=\bigoplus_j \mathrm{Hom}_B(\mathrm{Tot}(M^{0,\bullet,\bullet})^{-j},J\hat{\otimes}_AZ^{-j+1})$ 
represents a class in $F_{II}^1\,\HH^1(\mathrm{Tot}(L^{\bullet,\bullet}))=\psi_{II}^0(E_\infty^{1,0})$. 
One uses $\beta$ to construct a representative $(\xi,h_\xi)$ in $\tilde{F}_{II}^1$ such that
the corresponding element $\beta_\xi=(\beta_{\xi,j})\in L^{1,0}$ defined by $(\ref{eq:beta})$ satisfies
\begin{equation}
\label{eq:soomuchdumber}
[\beta]=[\beta_\xi]
\end{equation}
in $F_{II}^0\,\HH^1(\mathrm{Tot}(L^{\bullet,\bullet}))$.
To find $(\xi,h_\xi)$, one first shows that there exists an element 
$\gamma_\beta=(\gamma_{\beta,j})\in L^{0,0}$ with
\begin{equation}
\label{eq:soodumber}
[\beta]=[-d_I^{0,0}(\gamma_\beta)]
\end{equation}
in $F_{II}^0\,\HH^1(\mathrm{Tot}(L^{\bullet,\bullet}))$.
To define $\gamma_\beta$, let 
$f_\beta: \mathrm{Tot}(M^{0,\bullet,\bullet})\to J\hat{\otimes}_AZ^\bullet[1]$
be the map given by $f_\beta^{-j}=\beta_j$ for all $j$. Because $(\beta_j)\in L^{0,1}$, 
it follows that $f_\beta$ is a homomorphism in $C^-(B)$ that factors
through $\mathrm{Tot}(\HH_I^0(M^{\bullet,\bullet,\bullet}))=\mathrm{Tot}(M^{0,\bullet,\bullet})\,/
\,\mathrm{Tot}(\BB_I^0(M^{\bullet,\bullet,\bullet}))$. Let 
\begin{equation}
\label{eq:arrgghh!!}
\overline{f_\beta}:\mathrm{Tot}(\HH_I^0(M^{\bullet,\bullet,\bullet})) \to
J\hat{\otimes}_{A'}Z^\bullet[1]
\end{equation}
be the induced homomorphism in $C^-(B)$.
Since  $P^{0,\bullet}$ is acyclic, the morphism
$\alpha_1:Z^\bullet\to D^\bullet[1]$  from $(\ref{eq:alpha1})$ is an isomorphism in $D^-(B)$,
and we can use the projective Cartan-Eilenberg resolution $M^{\bullet,\bullet,\bullet}$
to represent the inverse of $\alpha_1$ by  a quasi-isomorphism
\begin{equation}
\label{eq:psi}
\psi_1:\mathrm{Tot}(\BB_I^0(M^{\bullet,\bullet,\bullet}))[1]\to
\mathrm{Tot}(\HH_I^0(M^{\bullet,\bullet,\bullet}))
\end{equation}
in $C^-(B)$. Define $\gamma_\beta=(\gamma_{\beta,j})\in L^{0,0}$
by
\begin{equation}
\label{eq:anotherL00}
\gamma_{\beta,j}=\overline{f_\beta}^{\,-j-1}\circ \psi_1^{-j-1}\circ\mathrm{proj}_{0,j}
\end{equation}
where $\mathrm{proj}_{0,j}$ is as in $(\ref{eq:L00})$.
Using $(\ref{eq:Ldifferential2})$, one checks that $[\beta]=[d_{II}^{0,0}(\gamma_\beta)]$
in $F_{II}^0\,\HH^1(\mathrm{Tot}(L^{\bullet,\bullet}))$,
which implies $(\ref{eq:soodumber})$.
Define $\hat{\alpha}_\beta:\mathrm{Tot}(\BB_I^0(M^{\bullet,\bullet,\bullet}))\to J\hat{\otimes}_AZ^\bullet$
in $C^-(B)$ by 
\begin{equation}
\label{eq:alphalpha!}
\hat{\alpha}_\beta=-\overline{f_\beta}[-1]\circ \psi_1[-1].
\end{equation}
It follows that
$\hat{\alpha}_\beta$ represents a morphism $\alpha\in\mathrm{Hom}_{D^-(B)}(D^\bullet,
J\hat{\otimes}_AZ^\bullet)$. By Lemma \ref{lem:gabberconstruct}(iii),
$\alpha\circ\tau$ defines a class  in $\tilde{F}_{II}^1$. Let $(\xi,h_\xi)$
be a representative of this class. 
Since $\eta_\xi=\eta_{\kappa_\xi}$ for $\kappa_\xi=\alpha\circ\tau$, one can take 
the morphism $\alpha_\xi$ from the beginning of the proof of part (iii) to be $\alpha_\xi=\alpha$.
This implies that in $(\ref{eq:kappaduo1})$ one can take
$\alpha_{\xi,\sharp}=\hat{\alpha}_\beta$.
Using $(\ref{eq:sooodumb})$ and comparing $\gamma_j$ in $(\ref{eq:L00})$ to $\gamma_{\beta,j}$
in $(\ref{eq:anotherL00})$, one sees $(\ref{eq:soomuchdumber})$.

Part (iv) follows from part (iii) above and from parts (ii) and (iii) of
Lemma \ref{lem:gabberconstruct}.
\end{proof}

%%%%%%%%%%%%%%%%%%%%%%%%%%%%%%%%%%%%%%%%%%%%%%%%%%%
%%Proof of Lemma $\ref{lem:extra}$ and Theorem $\ref{thm:obstructions}$
%%%%%%%%%%%%%%%%%%%%%%%%%%%%%%%%%%%%%%%%%%%%%%%%%%%

\subsection{Proof of Lemma $\ref{lem:extra}$ and Theorem $\ref{thm:obstructions}$}
\label{ss:obstruct}

In this subsection we prove Lemma $\ref{lem:extra}$ and Theorem $\ref{thm:obstructions}$  by proving 
Lemmas \ref{lem:takethisout} and \ref{lem:gabberlift} below. The proof relies on Lemmas 
\ref{lem:gabberconstruct} and \ref{lem:gabberfilter} and the following result.

\begin{lemma}
\label{lem:gabberlemma1}
$\mathrm{[O.\;Gabber]}$
Assume Hypothesis $\ref{hypo:obstruct}$, and suppose we have a short exact sequence in
$C^-(B')$
\begin{equation}
\label{eq:sequ}
\xi:\qquad 0\to X^\bullet \xrightarrow{u_\xi} Y^\bullet \xrightarrow{v_\xi} Z^\bullet\to 0
\end{equation}
where the terms of $X^\bullet$  are annihilated by $J$. 
Let $f_\xi:\HH^{-1}_I(A\hat{\otimes}_{A'}P^{\bullet,\bullet})=\mathrm{Tor}^{A'}_1(Z^\bullet,A)
\to X^\bullet$ be the homomorphism in $C^-(B)$ resulting from tensoring $\xi$ with $A$
over $A'$.
Then $f_\xi$ is an isomorphism in $D^-(B)$ if and only if 
the homomorphism $\upsilon:A\hat{\otimes}_{A'}^{\LL} Y^\bullet \to Z^\bullet$ 
induced by $A\hat{\otimes}_{A'}^{\LL}-$ is an isomorphism in $D^-(B)$.
\end{lemma}

\begin{rem}
\label{rem:specialups}
The homomorphism $\upsilon:A\hat{\otimes}_{A'}^{\LL} Y^\bullet \to Z^\bullet$ in $D^-(B)$
in Lemma \ref{lem:gabberlemma1}
is given as follows. Let $Q^\bullet$ be a bounded above complex 
of projective  pseudocompact $B'$-modules such that there is a quasi-isomorphism 
$\rho:Q^\bullet\to Y^\bullet$ in $C^-(B')$ that is surjective on terms. 
Then $\upsilon$ is represented in $D^-(B)$ by a homomorphism 
$\upsilon_Q:A\hat{\otimes}_{A'}Q^\bullet\to Z^\bullet$ in $C^-(B)$
which is the composition
\begin{equation}
\label{eq:tobeornottobe}
A\hat{\otimes}_{A'}Q^\bullet\xrightarrow{A\hat{\otimes}_{A'}\rho} 
A\hat{\otimes}_{A'}Y^\bullet\xrightarrow{A\hat{\otimes}_{A'}v_\xi}
A\hat{\otimes}_{A'}Z^\bullet=Z^\bullet.
\end{equation}
\end{rem}

\begin{proof}
Let $Q^\bullet$, $\rho$ and $\upsilon_Q$ be as in Remark \ref{rem:specialups} so that $\upsilon_Q$
represents $\upsilon$.
We obtain a commutative diagram in $C^-(B')$ with exact rows 
\begin{equation}
\label{eq:triangle00}
\xymatrix{
0\ar[r]&J\hat{\otimes}_{A'}Q^\bullet\ar[d]_{\mu_Y\circ(J\hat{\otimes}_{A'}\rho)}\ar[r]&Q^\bullet\ar[r]
\ar[d]^{\rho}&A\hat{\otimes}_{A'}Q^\bullet \ar[r]\ar[d]^{\upsilon_Q} &0\\
0\ar[r]&X^\bullet\ar[r]^{u_\xi} & Y^\bullet \ar[r]^{v_\xi}& Z^\bullet \ar[r]& 0}
\end{equation}
where $\mu_Y: J\hat{\otimes}_{A'}Y^\bullet\to X^\bullet$ is the composition of the natural
homomorphisms $J\hat{\otimes}_{A'}Y^\bullet \to JY^\bullet \to X^\bullet$.
By tensoring the diagram $(\ref{eq:triangle00})$ with $A$ over $A'$, and by also tensoring
$\upsilon_Q:A\hat{\otimes}_{A'}Q^\bullet\to Z^\bullet$ with $0\to J\to A'\to A\to 0$ over $A'$,
one sees that in $C^-(B)$
\begin{equation}
\label{eq:olala}
\mu_Y\circ (J\hat{\otimes}_{A'}\rho) = f_\xi\circ\iota^{-1}\circ (J\hat{\otimes}_A\upsilon_Q),
\end{equation}
where $\iota:\HH^{-1}_I(A\hat{\otimes}_{A'}P^{\bullet,\bullet})\to J\hat{\otimes}_{A}Z^\bullet$ 
is the isomorphism in $C^-(B)$ from Definition \ref{def:iota}. 

To prove the lemma, 
suppose first that $\upsilon$, and hence $\upsilon_Q$, is an isomorphism in $D^-(B)$. 
Since $\rho$ is a quasi-isomorphism in $C^-(B')$, one sees,
using $(\ref{eq:triangle00})$, that $\mu_Y\circ(J\hat{\otimes}_{A'}\rho)$
is a quasi-isomorphism in $C^-(B)$. By $(\ref{eq:olala})$, this implies that
$f_\xi$ is an isomorphism in $D^-(B)$.

Conversely, suppose  that $f_\xi$ is an isomorphism in $D^-(B)$. 
Rewriting $(\ref{eq:triangle00})$ with the aid of $(\ref{eq:olala})$, one obtains a
commutative diagram with exact rows in $C^-(B')$
\begin{equation}
\label{eq:triangle11}
\xymatrix{
0\ar[r] & J\hat{\otimes}_{A'}Q^\bullet \ar[r]\ar[d]_{J\hat{\otimes}_A\upsilon_Q}& 
Q^\bullet \ar[r]\ar[d]^{\rho} & 
A\hat{\otimes}_{A'}Q^\bullet\ar[r]\ar[d]^{\upsilon_Q}&0\\
&J\hat{\otimes}_AZ^\bullet\ar[r]^{u_\xi'}\ar[d]& Y^\bullet\ar[r]^{v_\xi}\ar[d]& Z^\bullet\ar[r]\ar[d]& 0\\
&C(J\hat{\otimes}_A\upsilon_Q)^\bullet\ar[r]&C(\rho)^\bullet\ar[r]&C(\upsilon_Q)^\bullet}
\end{equation}
where $u_\xi'=u_\xi\circ f_\xi\circ\iota^{-1}$. Because
$f_\xi\circ\iota^{-1}:J\hat{\otimes}_AZ^\bullet\to X^\bullet$ is  an isomorphism in $D^-(B')$, 
the rows in $(\ref{eq:triangle11})$ represent triangles in $D^-(B')$. 
Using the triangle corresponding to the last row in $(\ref{eq:triangle11})$, one argues 
inductively that $C(\upsilon_Q)^\bullet$ is acyclic.
To make this argument, one uses that $C(\rho)^\bullet$ is acyclic, that 
the terms of $C(\upsilon_Q)^\bullet$ are topologically free over $A$ and that all 
complexes involved are bounded above. The acyclicity of  
$C(\upsilon_Q)^\bullet$ implies that $\upsilon_Q$, and hence $\upsilon$, is an 
isomorphism in $D^-(B)$.
\end{proof}

We also need the following result which relates quasi-lifts of $(Z^\bullet,\zeta)$
over $A'$ to short exact sequences $\xi$ in $C^-(B')$ as in Definition $\ref{def:gabberclasses}$.

\begin{lemma}
\label{lem:takethisout}
Assume Hypothesis $\ref{hypo:obstruct}$ and the notation introduced in Definition
$\ref{def:gabberclasses}$.
Suppose $(Y^\bullet,\upsilon)$ is a quasi-lift  of $(Z^\bullet,\zeta)$ over $A'$.  Then
there exists a quasi-lift $({Y'}^\bullet,\upsilon')$ of $(Z^\bullet,\zeta)$ over $A'$
which is locally isomorphic to $(Y^\bullet,\upsilon)$ with the following properties:
\begin{itemize}
\item[(a)]
There is a short exact sequence
$\xi':0\to {X'}^\bullet\to {Y'}^\bullet\to Z^\bullet\to 0$ in $C^-(B')$ as in Definition 
$\ref{def:gabberclasses}$, i.e. the terms of ${X'}^\bullet$ are annihilated by $J$ and there is an 
isomorphism  ${X'}^\bullet\to J\hat{\otimes}_AZ^\bullet$ in $D^-(B)$.
\item[(b)] The isomorphism $\upsilon':A\hat{\otimes}^{\LL}_{A'}{Y'}^\bullet \to Z^\bullet$ is the
homomorphism in $D^-(B)$ from Lemma  $\ref{lem:gabberlemma1}$ 
which is induced by $A\hat{\otimes}^{\LL}_{A'}-$ relative to $\xi'$.
\end{itemize}
\end{lemma}

\begin{proof}
Using Theorem \ref{thm:derivedresult} and Remark \ref{rem:dumbdumb}, we may assume that the terms  $Y^i$ of $Y^\bullet$
are zero for $i<-p_0$ and $i>-1$, they are  projective pseudocompact 
$B'$-modules for $-p_0<i\le -1$, and $Y^{-p_0}$ is topologically free over $A'$.
Since the terms $Z^i$ of $Z^\bullet$ are projective pseudocompact $B$-modules
for $i>-p_0$, it follows that the inverse of the isomorphism
$\upsilon:A\hat{\otimes}_{A'}Y^\bullet\to Z^\bullet$ in $D^-(B)$ can be represented by
a quasi-isomorphism $\chi:Z^\bullet \to A\hat{\otimes}_{A'}Y^\bullet$ in $C^-(B)$.
We obtain a pullback diagram in $C^-(B')$ with exact rows
\begin{equation}
\label{eq:ohsomany}
\xymatrix{
0\ar[r] & J\hat{\otimes}_{A'}Y^\bullet \ar[r] &{Y'}^\bullet \ar[r] \ar[d]_{\chi_Y}  & Z^\bullet \ar[r]
\ar[d]^{\chi} & 0\\
0\ar[r]& J\hat{\otimes}_{A'}Y^\bullet \ar@{=}[u] \ar[r] & Y^\bullet \ar[r] & A\hat{\otimes}_{A'}Y^\bullet
\ar[r] & 0.}
\end{equation}
It follows that $\chi_Y$ is a quasi-isomorphism in $C^-(B')$. 
Letting ${X'}^\bullet = J\hat{\otimes}_{A'}Y^\bullet$, 
the top row of $(\ref{eq:ohsomany})$ defines a short exact sequence $\xi'$ as in part (a).
To prove part (b), let $\upsilon':A\hat{\otimes}^{\LL}_{A'}{Y'}^\bullet \to 
Z^\bullet$ be the homomorphism in $D^-(B)$ from Lemma  $\ref{lem:gabberlemma1}$, 
which is induced by $A\hat{\otimes}^{\LL}_{A'}-$ relative to the top row $\xi'$ of $(\ref{eq:ohsomany})$. 
By representing $\upsilon'$ by a homomorphism in $C^-(B)$
as in Remark \ref{rem:specialups}, one sees that in $D^-(B)$
$$A\hat{\otimes}_{A'}^{\LL}\chi_Y = \chi \circ \upsilon'= \upsilon^{-1}\circ \upsilon'.$$
Hence $\upsilon'$ is an isomorphism in $D^-(B)$, and $\chi_Y$ defines a local
isomorphism between the quasi-lifts $(Y^\bullet,\upsilon)$ and $({Y'}^\bullet,\upsilon')$
of $(Z^\bullet,\zeta)$ over $A'$. 
\end{proof}

\begin{lemma}
\label{lem:gabberlift}
Assume the notation of Definition $\ref{def:gabber}$, so that in particular, 
$P^{0,\bullet}$ is an acyclic
complex of projective pseudocompact $B'$-modules.
Let $\iota:\HH_I^{-1}(A\hat{\otimes}_{A'}P^{\bullet,\bullet})\to J\hat{\otimes}_A Z^\bullet$ be
the isomorphism in $C^-(B)$ 
from Definition $\ref{def:iota}$, so $\iota\in E_2^{0,1}$.
Let $\omega=\omega(Z^\bullet,A')$ be the class
$\omega=d_2^{0,1}(\iota)\in E_2^{2,0}=
\mathrm{Ext}^2_{D^-(B)}(Z^\bullet,J\hat{\otimes}_AZ^\bullet)$.
\begin{enumerate}
\item[(i)] Suppose $(Z^\bullet,\zeta)$ has a quasi-lift $(Y^\bullet,\upsilon)$ over $A'$.  Then $\omega=0$.
\item[(ii)] Conversely, suppose that $\omega=0$. 
\begin{enumerate}
\item[(a)]
There exists $\kappa\in\mathrm{Hom}_{D^-(B)}(
A\hat{\otimes}_{A'}T^\bullet,J\hat{\otimes}_AZ^\bullet)$ with $\kappa\circ\sigma = \iota$. 
Let $(\xi,h_\xi)$ represent the class $\eta_\kappa$ in $\tilde{F}_{II}^0$, as defined in 
Definition $\ref{def:pushout}$, where
$\xi:\;0\to X^\bullet\xrightarrow{u_\xi}Y^\bullet\xrightarrow{v_\xi}Z^\bullet\to 0$. 
Let $\upsilon:A\hat{\otimes}^{\LL}_{A'}Y^\bullet \to Z^\bullet$ be the homomorphism
in $D^-(B)$ from Lemma $\ref{lem:gabberlemma1}$  relative to $\xi$. Then 
$(Y^\bullet,\upsilon)$ is a quasi-lift of $(Z^\bullet,\zeta)$ over $A'$, which we denote by 
$(Y_\kappa^\bullet,\upsilon_\kappa)$. 
\item[(b)]
Let $\Xi$ be the set of the classes $\eta_\kappa$ in $\tilde{F}_{II}^0$ as
$\kappa$ varies over all choices of elements of 
$\mathrm{Hom}_{D^-(B)}(A\hat{\otimes}_{A'}T^\bullet, J\hat{\otimes}_AZ^\bullet)$
with $\kappa\circ\sigma=\iota$.
Then  the map $\eta_\kappa\mapsto [(Y_\kappa^\bullet,\upsilon_\kappa)]$ defines a bijection 
between $\Xi$ and the set $\Upsilon$ of all local isomorphism classes of quasi-lifts of
$(Z^\bullet,\zeta)$ over $A'$. 
\item[(c)] Let $[\iota]$ be the class of $\iota$ in $E_\infty^{0,1}$.
The set of all local isomorphism classes of quasi-lifts of $(Z^\bullet,\zeta)$ over 
$A'$ is in bijection with the full preimage of $[\iota]$ 
in $F_{II}^0\,\HH^1(\mathrm{Tot}(L^{\bullet,\bullet}))=\tilde{F}_{II}^0$ under $\varphi_{II}^0$.
In other words, the set of all local isomorphism classes of quasi-lifts of $(Z^\bullet,\zeta)$ over 
$A'$ is a principal homogeneous space for $E_\infty^{1,0}$. The set of all 
$\kappa\in\mathrm{Hom}_{D^-(B)}(A\hat{\otimes}_{A'}T^\bullet, J\hat{\otimes}_AZ^\bullet)$
with $\kappa\circ\sigma=\iota$ is a principal homogeneous space for $E_2^{1,0}=
\mathrm{Ext}^1_{D^-(B)}(Z^\bullet, J\hat{\otimes}_AZ^\bullet)$.
\item[(d)] We have $E_2^{p,0}=E_\infty^{p,0}$ for all $p$, i.e. the spectral sequence 
$(\ref{eq:spectral})$ partially degenerates. 
\end{enumerate}
\end{enumerate}
\end{lemma}

\begin{proof}
For part (i), suppose $(Y^\bullet,\upsilon)$ is a quasi-lift of $(Z^\bullet,\zeta)$ over $A'$. 
Using Theorem \ref{thm:derivedresult}, we may assume that the
terms of $Y^\bullet$ are projective
pseudocompact $B'$-modules. Moreover, by adding an acyclic complex of topologically
free pseudocompact $B'$-modules to $Y^\bullet$ if necessary, we can assume that
$\upsilon:A\hat{\otimes}_{A'}Y^\bullet\to Z^\bullet$ is given by a quasi-isomorphism of complexes
in $C^-(B)$ that is surjective on terms. Hence we a have a short exact sequence
in $C^-(B')$ of the form
\begin{equation}
\label{eq:latergater1off}
 0\to K^\bullet\to Y^\bullet \xrightarrow{v_Y} Z^\bullet\to 0,
\end{equation}
where $v_Y$ is the composition $Y^\bullet\to A\hat{\otimes}_{A'}Y^\bullet\xrightarrow{\upsilon}
Z^\bullet$ and $K^\bullet=\mathrm{Ker}(v_Y)$. Note that
$K^\bullet$ may or may not be annihilated by $J$. Since $P^{0,\bullet}$ is a projective
object in $C^-(B')$, we obtain
a commutative diagram in $C^-(B')$ 
whose top (resp. bottom) row is given by $(\ref{eq:sequ1})$ (resp. $(\ref{eq:latergater1off})$).
Tensoring this diagram with $A$ over $A'$, we get a commutative diagram in $C^-(B)$ 
with exact rows 
\begin{equation}
\label{eq:diagg2alt}
\xymatrix {
0\ar[r] & \HH_I^{-1}(A\hat{\otimes}_{A'}P^{\bullet,\bullet})\ar@{=}[d]\ar[r]^(.58)\sigma&
A\hat{\otimes}_{A'}T^\bullet\ar[d]\ar[r]^(.48){A\hat{\otimes}_{A'}\delta}
& A\hat{\otimes}_{A'}P^{0,\bullet}\ar[d]\ar[r]^(.65){A\hat{\otimes}_{A'}\epsilon} 
& Z^\bullet\ar@{=}[d]\ar[r]&0\\
0\ar[r]& \HH_I^{-1}(A\hat{\otimes}_{A'}P^{\bullet,\bullet})\ar[r]^(.58){f_Y}& 
A\hat{\otimes}_{A'}K^\bullet\ar[r]&A\hat{\otimes}_{A'}Y^\bullet\ar[r]^(.6){\upsilon}&Z^\bullet\ar[r]&0
}
\end{equation}
whose top row is given by  $(\ref{eq:sequ2})$. Using the definition of $\alpha_1$ and
$\alpha_2$ in $(\ref{eq:alpha1})$ and $(\ref{eq:alpha2})$, one sees
that the top row of $(\ref{eq:diagg2alt})$ defines the class $\alpha_2[1]\circ\alpha_1$ in
$\mathrm{Ext}^2_{D^-(B)}(Z^\bullet,\HH_I^{-1}(A\hat{\otimes}_{A'}P^{\bullet,\bullet}))$.
Because $\upsilon$ is an isomorphism
in $D^-(B)$, the bottom row of $(\ref{eq:diagg2alt})$ shows that $f_Y$ is also an isomorphism in 
$D^-(B)$. Therefore, $\alpha_2[1]\circ\alpha_1=0$ in $D^-(B)$. Since 
$\omega=d_2^{0,1}(\iota)=\iota[2]\circ \alpha_2[1]\circ\alpha_1$
by Lemma \ref{lem:oyoyoy!}, this implies $\omega=0$  in $D^-(B)$.

For part (ii), assume that $\omega=d_2^{0,1}(\iota)=0$. 
By Lemma \ref{lem:gabberconstruct}(ii), there exists 
$\kappa:A\hat{\otimes}_{A'}T^\bullet\to J\hat{\otimes}_A Z^\bullet$ 
in $D^-(B)$ with $\kappa\circ\sigma=\iota$.
Let $(\xi,h_\xi)$ and $\upsilon$ be as in the statement of part (ii)(a),
where $\xi:\;0\to X^\bullet\xrightarrow{u_\xi}Y^\bullet\xrightarrow{v_\xi}Z^\bullet\to 0$.
Let $f_\xi: \HH_I^{-1}(A\hat{\otimes}_{A'}P^{\bullet,\bullet})\to X^\bullet$ be 
the homomorphism in $C^-(B)$ resulting from tensoring $\xi$ with $A$ over $A'$.
By Lemma \ref{lem:gabberconstruct}(ii),
$h_\xi\circ f_\xi=\iota$, which
implies that $f_\xi$ is an isomorphism in $D^-(B)$. Hence by Lemma 
\ref{lem:gabberlemma1},  $\upsilon:A\hat{\otimes}^{\LL}_{A'}Y^\bullet \to Z^\bullet$ 
is an isomorphism in $D^-(B)$.
Using the isomorphism $\upsilon$ together with the fact that 
$Z^\bullet$ has finite pseudocompact $A$-tor dimension, it follows that
there exists an integer $N$ such that $\HH^i(S\hat{\otimes}^{\LL}_{A'}Y^\bullet)=0$
for all $i<N$ and for all pseudocompact $A$-modules $S$.
Since for all pseudocompact $A'$-modules $S'$ we have that 
$JS'$ and $S'/JS'$ are annihilated by $J$ and thus pseudocompact $A$-modules,
one sees that $\HH^i(S'\hat{\otimes}^{\LL}_{A'}Y^\bullet)=0$ for all $i<N$.
Hence $(Y^\bullet,\upsilon)$ is a quasi-lift of $(Z^\bullet,\zeta)$ over $A'$, which we denote by
$(Y_\kappa^\bullet,\upsilon_\kappa)$.

Let $\Xi$ and $\Upsilon$ be as in the statement of part (ii)(b).
We need to show that the map
\begin{eqnarray}
\label{eq:TheMap}
\Xi &\to &\Upsilon\\
\eta_\kappa &\mapsto& [(Y_\kappa^\bullet,\upsilon_\kappa)]\nonumber
\end{eqnarray}
is a bijection. 
This map is well-defined, since, as seen at the end of Definition \ref{def:pushout}, 
$\eta_\kappa=\eta_{\kappa'}$ 
if and only if $\kappa\circ a_T=\kappa'\circ a_T$ in $D^-(B')$
and the construction in Definition \ref{def:pushout} shows that $\kappa\circ a_T$
determines the local isomorphism class $[(Y_\kappa^\bullet,\upsilon_\kappa)]$.

We first prove that $(\ref{eq:TheMap})$ is surjective. Given a quasi-lift
$(Y^\bullet,\upsilon)$ of $(Z^\bullet,\zeta)$ over $A'$, 
we may assume by Lemma \ref{lem:takethisout} that there is a short exact sequence
$\xi:0\to X^\bullet\xrightarrow{u_\xi} Y^\bullet\xrightarrow{v_\xi} Z^\bullet\to 0$ in $C^-(B')$ as in 
Definition $\ref{def:gabberclasses}$ and that
the isomorphism $\upsilon:A\hat{\otimes}^{\LL}_{A'}Y^\bullet \to Z^\bullet$ is the
homomorphism in $D^-(B)$ from Lemma  $\ref{lem:gabberlemma1}$ relative to $\xi$.
Since $\upsilon$ is an isomorphism in $D^-(B)$, it follows from Lemma \ref{lem:gabberlemma1} that
the homomorphism $f_\xi:\HH_I^{-1}(A\hat{\otimes}_{A'}P^{\bullet,\bullet})\to X^\bullet$ 
is an isomorphism in $D^-(B)$. Letting $h_\xi=\iota\circ f_\xi^{-1}$, it follows that $(\xi,h_\xi)$
represents a class $\eta_\xi$ in $\tilde{F}_{II}^0$.
By Lemma \ref{lem:gabberconstruct}(i), there exists 
$\kappa\in\mathrm{Hom}_{D^-(B)}(A\hat{\otimes}_{A'}T^\bullet,J\hat{\otimes}_AZ^\bullet)$ 
such that $\kappa\circ\sigma=h_\xi\circ f_\xi=\iota$ and $\eta_\xi=\eta_{\kappa}$ in $\tilde{F}_{II}^0$.
Following the definition of $(Y_\kappa^\bullet,\upsilon_\kappa)$, one sees that
$(Y_\kappa^\bullet,\upsilon_\kappa)$ and $(Y^\bullet,\upsilon)$ are locally isomorphic
quasi-lifts of $(Z^\bullet,\zeta)$ over $A'$. 

To prove that $(\ref{eq:TheMap})$ is injective, let $\eta_\kappa,\eta_{\kappa'}\in \Xi$ 
be such that
$(Y_\kappa^\bullet,\upsilon_\kappa)$ and $(Y^\bullet_{\kappa'},\upsilon_{\kappa'})$
are locally isomorphic quasi-lifts of $(Z^\bullet,\zeta)$ over $A'$. This means that
there exists an isomorphism $\theta:Y_\kappa^\bullet\to Y_{\kappa'}^\bullet$ in $D^-(B')$ with
$\upsilon_{\kappa'}\circ (A\hat{\otimes}^{\LL}_{A'}\theta) = \upsilon_\kappa$.
Consider the triangle in $D^-(B')$
\begin{equation}
\label{eq:undnocheins}
J\hat{\otimes}_{A'}^{\LL}Y_{\kappa}^\bullet\to A'\hat{\otimes}_{A'}^{\LL}Y_{\kappa}^\bullet\to 
A\hat{\otimes}_{A'}^{\LL}Y_{\kappa}^\bullet \xrightarrow{\eta^{\LL}_{\kappa}} 
J\hat{\otimes}_{A'}^{\LL}Y_{\kappa}^\bullet[1],
\end{equation}
which is associated to the short exact sequence obtained by 
applying the functor $-\hat{\otimes}_{A'}^{\LL}Y_{\kappa}^\bullet$ to the sequence
$0\to J \to A'\to A\to 0$.
Using the definition of $\upsilon_{\kappa}$, one sees that
$\eta_{\kappa}\circ\upsilon_{\kappa}=(J\hat{\otimes}_A^{\LL}
\upsilon_{\kappa}[1])\circ \eta^{\LL}_{\kappa}$.
On replacing $\kappa$ by $\kappa'$, one obtains a similar equation relating  $\eta_{\kappa'}$
and $\eta^{\LL}_{\kappa'}$. Since
$(J\hat{\otimes}_{A'}^{\LL}\theta[1])\circ \eta^{\LL}_{\kappa}
=\eta^{\LL}_{\kappa'}\circ (A\hat{\otimes}_{A'}^{\LL}\theta)$, this implies that 
$\eta_\kappa=\eta_{\kappa'}$.

The first statement of part (ii)(c) follows from part (ii)(b) above and from Lemma \ref{lem:gabberfilter}(iv).
For the second statement of part (ii)(c), one notes that since 
$\omega=\iota[2]\circ\alpha_2[1]\circ\alpha_1=0$ and
$\iota$ and $\alpha_1$ are isomorphisms in $D^-(B)$, one has  $\alpha_2=0$. 
Replacing $\alpha_2=0$ in the triangle $(\ref{eq:alpha2})$ and applying the
functor $\mathrm{Hom}_{D^-(B)}(-,J\hat{\otimes}_AZ)$, one obtains 
a short exact sequence of abelian groups
$$0\to 
\mathrm{Hom}(D^\bullet, J\hat{\otimes}_AZ^\bullet)\xrightarrow{\tau^*}
\mathrm{Hom}( A\hat{\otimes}_{A'}T^\bullet, J\hat{\otimes}_AZ^\bullet)\xrightarrow{\sigma^*}
\mathrm{Hom}( \HH_I^{-1}(A\hat{\otimes}_{A'}P^{\bullet,\bullet}), J\hat{\otimes}_AZ^\bullet)\to 0,$$
where $\mathrm{Hom}$ stands for $\mathrm{Hom}_{D^-(B)}$.
Since $\mathrm{Hom}_{D^-(B)}(D^\bullet, J\hat{\otimes}_AZ^\bullet)\cong
\mathrm{Ext}^1_{D^-(B)}(Z^\bullet, J\hat{\otimes}_AZ^\bullet)$, part (ii)(c) follows.

To prove part (ii)(d), we show that for all $p$ the inflation map
$$\mathrm{Inf}_{B}^{B'}:\mathrm{Ext}^p_{D^-(B)}(Z^\bullet,J\hat{\otimes}_AZ^\bullet)
\to \mathrm{Ext}^p_{D^-(B')}(Z^\bullet,J\hat{\otimes}_AZ^\bullet)$$
is injective, which implies that $E_\infty^{p,0}=E_2^{p,0}=
\mathrm{Ext}^p_{D^-(B)}(Z^\bullet,J\hat{\otimes}_AZ^\bullet)$.
Let $(Y^\bullet,\upsilon)$ be a quasi-lift of
$(Z^\bullet,\zeta)$ such that $Y^\bullet$ is a bounded above complex of topologically free
pseudocompact $B'$-modules. Let $a_Y:Y^\bullet\to A\hat{\otimes}_{A'}Y^\bullet$ be the 
natural homomorphism in $C^-(B')$, and let 
$\pi_P:\mathrm{Tot}(P^{\bullet,\bullet})\to Z^\bullet$ be the quasi-isomorphism in 
$C^-(B')$ from Definition \ref{def:spectralseqA}. Then $g=\pi_P^{-1}\circ \upsilon \circ a_Y\in
\mathrm{Hom}_{D^-(B')}(Y^\bullet,\mathrm{Tot}(P^{\bullet,\bullet}))=
\mathrm{Hom}_{K^-(B')}(Y^\bullet,\mathrm{Tot}(P^{\bullet,\bullet}))$. Suppose
$f\in \mathrm{Ext}^p_{D^-(B)}(Z^\bullet,J\hat{\otimes}_AZ^\bullet)$ and $\mathrm{Inf}_{B}^{B'}(f)=0$
in $\mathrm{Ext}^p_{D^-(B')}(Z^\bullet,J\hat{\otimes}_AZ^\bullet)$. Since 
$A\hat{\otimes}_{A'}Y^\bullet$ is a bounded above complex of topologically free 
pseudocompact $B$-modules, it follows that
$f\circ \upsilon\in \mathrm{Hom}_{K^-(B)}(A\hat{\otimes}_{A'}Y^\bullet,J\hat{\otimes}_AZ^\bullet[p])$. 
Since $\mathrm{Inf}_{B}^{B'}(f)=0$ and $\pi_P$ is a quasi-isomorphism in $C^-(B')$, it follows that  
$F=f\circ \pi_P:\mathrm{Tot}(P^{\bullet,\bullet})\to J\hat{\otimes}_AZ^\bullet[p]$
is homotopic to zero in $C^-(B')$. 
Then $(f\circ \upsilon) \circ a_Y = (f\circ \pi_P)\circ (\pi_P^{-1}\circ \upsilon\circ a_Y)
=(f\circ \pi_P)\circ g = F\circ g$, which implies that $(f\circ \upsilon)\circ a_Y$ is homotopic to
zero in $C^-(B')$.
Applying $A\hat{\otimes}_{A'}-$ shows that $f\circ \upsilon$ is homotopic to zero in $C^-(B)$.
Since $\upsilon$ is an isomorphism in $D^-(B)$ and
$\mathrm{Hom}_{K^-(B)}(A\hat{\otimes}_{A'}Y^\bullet,J\hat{\otimes}_AZ^\bullet[p])=
\mathrm{Hom}_{D^-(B)}(A\hat{\otimes}_{A'}Y^\bullet,J\hat{\otimes}_AZ^\bullet[p])$,
it follows that $f=0$ in $D^-(B)$ which proves part (ii)(d).
\end{proof}

\begin{rem}
\label{rem:notalwaysdegenerate}
If $\omega=\omega(Z^\bullet,A')\neq 0$, i.e. if there is no quasi-lift of $(Z^\bullet,\zeta)$ over $A'$, then
$E_\infty^{1,0}$ is a proper quotient of $E_2^{1,0}$ in general. 
For example, let $k=\mathbb{Z}/2$, $A=k[t]/(t^4)$, $A'=k[t]/(t^6)$ and let 
$\pi:A'\to A$ be the natural surjection. Let $G$ be the trivial group,
so that $B=A$ and $B'=A'$. Suppose $V^\bullet = k\xrightarrow{0} k \xrightarrow{0} k$
and $Z^\bullet = A\xrightarrow{t^3} A \xrightarrow{t} A$ are both concentrated in degrees $-3,-2,-1$.
Then $J\hat{\otimes}_AZ^\bullet=A/t^2A \xrightarrow{0} A/t^2A\xrightarrow{t} A/t^2A$
is also concentrated in degrees $-3,-2,-1$. 
We now show that the inflation map 
$$\mathrm{Inf}_{A}^{A'}:
\mathrm{Ext}^1_{D^-(A)}(Z^\bullet,J\hat{\otimes}_AZ^\bullet)
\to \mathrm{Ext}^1_{D^-(A')}(Z^\bullet,J\hat{\otimes}_AZ^\bullet)$$
is not injective.  This implies that $E_\infty^{1,0}$ is a proper quotient of $E_2^{1,0}  = \mathrm{Ext}^1_{D^-(A)}(Z^\bullet,J\hat{\otimes}_AZ^\bullet)$,
since $A=B$ and $A'=B'$ and $E_\infty^{1,0}$ is isomorphic to the image of $\mathrm{Inf}_{A}^{A'}$. Consider the map of complexes $f:Z^\bullet \to
J\hat{\otimes}_AZ^\bullet[1]$ in $C^-(A)$ where $f^j=0$ for all $j\neq -3$ and
$f^{-3}:Z^{-3}=A \to A/t^2A=J\hat{\otimes}_AZ^{-2}$
sends $1\in A$ to $t\in A/t^2A$. Then $f$ is not homotopic to zero which implies that
$f$ is not zero in $D^-(A)$ since the terms of $Z^\bullet$ are topologically free 
pseudocompact $A$-modules. To show that $\mathrm{Inf}_{A}^{A'}(f)=0$ in $D^-(A')$,
we construct a suitable bounded above complex $Q^\bullet$ of topologically free pseudocompact  
$A'$-modules 
together with a quasi-isomorphism $s_Q:Q^\bullet\to Z^\bullet$. Namely, let
$$Q^\bullet: \quad\cdots \;
(A')^2 \xrightarrow{\footnotesize \left(\begin{array}{cc}t^3&0\\0&t^5\end{array}\right)}
(A')^2 \xrightarrow{\footnotesize \left(\begin{array}{cc}t^3 & 0\\0&0\\0&t\end{array}\right)}
(A')^3\xrightarrow{\footnotesize \left(\begin{array}{ccc}0&1&0\\0&0&t^5\end{array}\right)} 
(A')^2 \xrightarrow{\footnotesize \left(\begin{array}{cc}0&t\end{array}\right)} A' $$
be concentrated in degrees $\le -1$, and let $s_Q=\left(s_Q^j\right)$ where
$s_Q^j=0$ for $j\not\in \{-3,-2,-1\}$ and
$s_Q^{-1}=\pi$, $s_Q^{-2}=(t^3\,\pi,\pi)$, $s_Q^{-3}=(t\,\pi,\pi,0)$.
It follows that $f\circ s_Q$ is homotopic to zero, and hence equal to zero in
$D^-(A')$, by defining $h^j:Q^j\to J\hat{\otimes}_AZ^j[1]=J\hat{\otimes}_AZ^{j+1}$ by
$h^j=0$ for all $j\neq -2$ and $h^{-2}=(t\,\overline{\pi},0)$ where $\overline{\pi}:A'\to A/t^2$ is the 
natural surjection. Since $s_Q$ is an isomorphism in $D^-(A')$, this implies that
$\mathrm{Inf}_{A}^{A'}(f) = (f\circ s_Q)\circ (s_Q)^{-1}$ is zero in $D^-(A')$.
\end{rem}

%%%%%%%%%%%%%%%%%%%%%%%%%%%%%%%%%%%%%%%%%%%%%%%%%%%
%%Proof of Lemma $\ref{lem:aut}$
%%%%%%%%%%%%%%%%%%%%%%%%%%%%%%%%%%%%%%%%%%%%%%%%%%%

\subsection{Proof of Lemma $\ref{lem:aut}$}
\label{ss:aut}

As in the statement of Lemma \ref{lem:aut}, suppose that $(Y^\bullet,\upsilon)$ is a 
quasi-lift of $(Z^\bullet,\zeta)$ over $A'$. 
Using Theorem \ref{thm:derivedresult}, we may assume that the
terms of $Y^\bullet$ are projective pseudocompact $B'$-modules.
Consider the triangle in $D^-(B')$
\begin{equation}
\label{eq:triang}
A\hat{\otimes}_{A'}Y^\bullet[-1]\xrightarrow{a} J\hat{\otimes}_{A'}Y^\bullet
\xrightarrow{b}  Y^\bullet\xrightarrow{c}
A\hat{\otimes}_{A'} Y^\bullet,
\end{equation}
which is associated to the short exact sequence obtained by 
applying the functor $-\hat{\otimes}_{A'}^{\LL}Y^\bullet=-\hat{\otimes}_{A'}Y^\bullet$ 
to the sequence $0\to J \to A'\to A\to 0$.
Applying the functor $\mathrm{Hom}_{D^-(B')}(Y^\bullet,-)$ to the triangle 
$(\ref{eq:triang})$, one obtains a long exact Hom sequence
\begin{equation}
\label{eq:homseq}
\mbox{\small $\xymatrix {
\cdots\ar[r] &\mathrm{Hom}_{D^-(B')}(Y^\bullet,Y^\bullet[-1])\ar[r]&
\mathrm{Hom}_{D^-(B')}(Y^\bullet,A\hat{\otimes}_{A'}Y^\bullet[-1])\ar[dll]_(.55){(a)_*}\\
\mathrm{Hom}_{D^-(B')}(Y^\bullet,J\hat{\otimes}_{A'}Y^\bullet)\ar[r]^(.58){(b)_*} & 
\mathrm{Hom}_{D^-(B')}(Y^\bullet,Y^\bullet) \ar[r]^{(c)_*}&
\mathrm{Hom}_{D^-(B')}(Y^\bullet,A\hat{\otimes}_{A'}Y^\bullet)\ar[r]& \cdots
}$}
\end{equation}
Using that $\mathrm{Image}((b)_*)$ is a two-sided ideal with square $0$ in
$\mathrm{Hom}_{D^-(B')}(Y^\bullet,Y^\bullet)$, one sees that
\begin{equation}
\label{eq:ouchie1}
\mathrm{Aut}^0_{D^-(B')}(Y^\bullet)\cong 
\mathrm{Image}((b)_*)\cong
\mathrm{Hom}_{D^-(B')}(Y^\bullet,J\hat{\otimes}_{A'}Y^\bullet)/\mathrm{Image}((a)_*).
\end{equation}
Since $Y^\bullet$  is a bounded above complex of projective 
pseudocompact $B'$-modules, 
$c$ induces an isomorphism 
$(c)^*: \mathrm{Hom}_{D^-(B)}(A\hat{\otimes}_{A'}Y^\bullet,W^\bullet)
\xrightarrow{\cong} \mathrm{Hom}_{D^-(B')}(Y^\bullet,W^\bullet)$
for all complexes $W^\bullet$ in $C^-(B)$.
Thus $(\ref{eq:ouchie1})$ implies 
\begin{equation}
\label{eq:ouchie2}
\mathrm{Aut}^0_{D^-(B')}(Y^\bullet)\cong 
\mathrm{Hom}_{D^-(B)}(A\hat{\otimes}_{A'}Y^\bullet,J\hat{\otimes}_{A'}Y^\bullet)/
\mathrm{Image}(\mathrm{Ext}^{-1}_{D^-(B)}(A\hat{\otimes}_{A'}Y^\bullet,A\hat{\otimes}_{A'}Y^\bullet)),
\end{equation}
where $\mathrm{Image}(\mathrm{Ext}^{-1}_{D^-(B)}(A\hat{\otimes}_{A'}Y^\bullet,
A\hat{\otimes}_{A'}Y^\bullet))$ is the image of
$\mathrm{Hom}_{D^-(B)}(A\hat{\otimes}_{A'}Y^\bullet,A\hat{\otimes}_{A'}Y^\bullet[-1])$
in $\mathrm{Hom}_{D^-(B)}(A\hat{\otimes}_{A'}Y^\bullet,J\hat{\otimes}_{A'}Y^\bullet)$
under the composition $((c)^*)^{-1}\circ (a)_*\circ (c)^*$.
Since $\upsilon$ induces an isomorphism
$J\hat{\otimes}_A\upsilon: J\hat{\otimes}_{A'}Y^\bullet\to J\hat{\otimes}_AZ^\bullet$ 
in $D^-(B)$,  Lemma \ref{lem:aut} follows.

%%%%%%%%%%%%%%%%%%%%%%%%%%%%%%%%%%%%%%%%%%%%%%%%%%%
%%Proof of Proposition $\ref{prop:compare}$
%%%%%%%%%%%%%%%%%%%%%%%%%%%%%%%%%%%%%%%%%%%%%%%%%%%

\subsection{Proof of Proposition $\ref{prop:compare}$}
\label{ss:compare}

As in the statement of Proposition \ref{prop:compare}, we assume the notation of
\S\ref{ss:naive} and Theorem \ref{thm:obstructions}. For simplicity, we 
identify $A\hat{\otimes}_{A'}Y^j=\tilde{Z}^j$ for all $j$ and we
identify $Z^\bullet$ with the truncation $\mathrm{Trunc}_{-p_0}(\tilde{Z}^\bullet)$
of $\tilde{Z}^\bullet$ at $-p_0$ which is obtained from
$\tilde{Z}^\bullet$ by replacing $\tilde{Z}^{-p_0}$ by 
$\tilde{Z}^{-p_0}/\mathrm{Image}(d_{\tilde{Z}}^{-p_0-1})$ and $\tilde{Z}^j$ be $0$ for all $j<-p_0$.
Let $s_Z:\tilde{Z}^\bullet\to Z^\bullet$ be the resulting quasi-isomorphism where
$s_Z^{-p_0}:\tilde{Z}^{-p_0}\to \tilde{Z}^{-p_0}/\mathrm{Image}(d_{\tilde{Z}}^{-p_0-1})=Z^{-p_0}$
is the natural surjection.

To be able to compare the two lifting obstructions $\omega(Z^\bullet, A')$ and 
$\omega_0(Z^\bullet,Z')$, we define a particular $P^{0,\bullet}$ and a particular
$\epsilon:P^{0,\bullet}\to Z^\bullet$ as in Definition \ref{def:gabber} by using 
$(Y^j,c_Y^j)$ from \S\ref{ss:naive}. By following Grothendieck's construction
discussed in Remark \ref{rem:projectivegrothendieck}, we define
$$P^{0,0}=Y^{-1}, \quad P^{0,-j}=Y^{-j-1}\oplus Y^{-j} \;(1\le j\le p_0-1), \quad P^{0,-p_0}=Y^{-p_0}$$
and the differentials as
$$d_{P^{0,\bullet}}^{-1}=(-c_Y^{-2},1), \quad d_{P^{0,\bullet}}^{-j}=
\left(\begin{array}{cc}-c_Y^{-j-1}&1\\-c_Y^{-j}\circ c_Y^{-j-1} & c_Y^{-j}\end{array}\right)
\;(2\le j\le p_0-1), \quad d_{P^{0,\bullet}}^{-p_0}=\left(\begin{array}{c}1\\ c_Y^{-p_0}\end{array}\right).$$
Moreover, we define $\epsilon$ by
$$\epsilon^0=0,\quad \epsilon^{-j} = (0,a_Y^{-j}) \; (1\le j\le p_0-1),\quad
\epsilon^{-p_0}=s_Z^{-p_0}\circ a_Y^{-p_0}$$
where $a_Y^{-j}:Y^{-j}\to A\hat{\otimes}_{A'}\tilde{Z}^{-j}$ is the natural surjection for $1\le j\le p_0$.

Following Definition \ref{def:gabber}, one now computes explicitly $T^\bullet=\mathrm{Ker}(\epsilon)$ 
and $D^\bullet=\mathrm{Ker}(A\hat{\otimes}_{A'}\epsilon)$ and identifies
$\HH_I^{-1}(A\hat{\otimes}_{A'}P^{\bullet,\bullet})$ with the kernel of the
surjection $\tau:A\hat{\otimes}_{A'}T^\bullet\to D^\bullet$. This computation shows that
$\HH_I^{-1}(A\hat{\otimes}_{A'}P^{\bullet,\bullet})$ can be identified with the truncation 
$\mathrm{Trunc}_{-p_0}(JY^\bullet)$ of the
complex $JY^\bullet$ at $-p_0$ which is obtained from $JY^\bullet$ by replacing $JY^{-p_0}$ by 
$JY^{-p_0}/\mathrm{Image}(d_{JY}^{-p_0-1})$ and $JY^j$ by $0$ for all $j<-p_0$.

We use the definition of $\omega(Z^\bullet,A')=d_2^{0,1}(\iota)$ in Theorem \ref{thm:obstructions}
which is by Lemma \ref{lem:oyoyoy!} equal to 
$$\omega(Z^\bullet,A') = \iota[2]\circ\alpha_2[1]\circ\alpha_1$$
where $\alpha_1$ and $\alpha_2$ are the homomorphisms in $D^-(B)$ which occur in the
triangles $(\ref{eq:alpha1})$ and $(\ref{eq:alpha2})$ in Definition \ref{def:gabber}. 
Using the mapping cones of the homomorphisms
$\delta_D:D^\bullet \to A\hat{\otimes}_{A'}T^\bullet$ and 
$\sigma:\HH_I^{-1}(A\hat{\otimes}_{A'}P^{\bullet,\bullet})\to A\hat{\otimes}_{A'}T^\bullet$ in 
$(\ref{eq:alpha1})$ and $(\ref{eq:alpha2})$, respectively, one sees that one can express 
$\alpha_2[1]\circ\alpha_1\in
\mathrm{Hom}_{D^-(B)}(Z^\bullet,\HH_I^{-1}(A\hat{\otimes}_{A'}P^{\bullet,\bullet})[2])$
as
\begin{equation}
\alpha_2[1]\circ\alpha_1= s_J[2]\circ \tilde{\omega}\circ (s_Z)^{-1}
\end{equation}
where $(s_Z)^{-1}$ is the inverse in $D^-(B)$ of the quasi-isomorphism $s_Z$, $\tilde{\omega}$
is as in $(\ref{eq:yuckyuck})$ and 
\begin{equation}
\label{eq:sJ}
s_J: JY^\bullet \to \mathrm{Trunc}_{-p_0}(JY^\bullet)=\HH_I^{-1}(A\hat{\otimes}_{A'}P^{\bullet,\bullet})
\end{equation}
is the quasi-isomorphism in $C^-(B)$ resulting from truncation such that 
$s_J^{-p_0}$ is the natural surjection.
It follows that
$$\omega(Z^\bullet,A') = \iota[2]\circ\alpha_2[1]\circ\alpha_1=
(\iota\circ s_J)[2] \circ \tilde{\omega} \circ (s_Z)^{-1}$$
in $D^-(B)$, which proves the first part of Proposition \ref{prop:compare}.

For the second part of Proposition \ref{prop:compare}, let $(Y_0^\bullet,\upsilon_0)$ and
$({Y'}^\bullet,\upsilon')$ be two quasi-lifts of $(Z^\bullet,\zeta)$ over $A'$.
Without loss of generality, we can assume that $Y_0^j=Y^j={Y'}^j$ for all $j$, 
by using a fixed versal deformation $(U^\bullet,\phi_U)$ of $V^\bullet$ over $R=R(G,V^\bullet)$
such that $U^\bullet$ is concentrated in degrees $\le -1$ and all terms of $U^\bullet$
are topologically free pseudocompact $R[[G]]$-modules. 
In particular, this implies that $JY_0^\bullet=JY^\bullet=J{Y'}^\bullet$ and
$\tilde{Z}^\bullet = A\hat{\otimes}_{A'}Y_0^\bullet = A\hat{\otimes}_{A'}{Y'}^\bullet$  
as complexes in $C^-(B)$. We have short exact sequences in $C^-(B')$ of the form
$0\to JY^\bullet \to Y_0^\bullet \xrightarrow{a_{Y_0}} \tilde{Z}^\bullet  \to 0$ and
$0\to JY^\bullet \to {Y'}^\bullet \xrightarrow{a_{Y'}} \tilde{Z}^\bullet \to 0$.
Truncating these complexes at $-p_0$ in the same way as we have done several times above
and using that we have assumed that $Z^\bullet=\mathrm{Trunc}_{-p_0}(\tilde{Z}^\bullet)$,
we obtain short exact sequences in $C^-(B')$ of the form
\begin{eqnarray}
\label{eq:firstone}
\xi_0:&0\to \mathrm{Trunc}_{-p_0}(JY^\bullet) \to \mathrm{Trunc}_{-p_0}(Y_0^\bullet) 
\xrightarrow{\mathrm{Trunc}_{-p_0}(a_{Y_0})} Z^\bullet  \to 0,&\\
\label{eq:secondone}
\xi':&0\to \mathrm{Trunc}_{-p_0}(JY^\bullet) \to \mathrm{Trunc}_{-p_0}({Y'}^\bullet) 
\xrightarrow{\mathrm{Trunc}_{-p_0}(a_{Y'})} Z^\bullet  \to 0.&
\end{eqnarray}
Since we have seen above that $\HH_I^{-1}(A\hat{\otimes}_{A'}P^{\bullet,\bullet})$ can be identified 
with $\mathrm{Trunc}_{-p_0}(JY^\bullet)$, letting $h_{\xi_0}=h_{\xi'}=\iota$ we arrive
at the class $\eta_{\xi_0}$ (resp. $\eta_{\xi'}$) in 
$\tilde{F}_{II}^0=F_{II}^0\,\HH^1(\mathrm{Tot}(L^{\bullet,\bullet}))$ represented by
$(\xi_0,h_{\xi_0})$ (resp. $(\xi',h_{\xi'})$) as described in Definition \ref{def:gabberclasses}.
It follows from Lemma \ref{lem:gabberlift}, parts (ii)(a) and (ii)(b), that $\eta_{\xi_0}$ (resp. $\eta_{\xi'}$)
is the class in $\tilde{F}_{II}^0=F_{II}^0\,\HH^1(\mathrm{Tot}(L^{\bullet,\bullet}))$ corresponding
to the local isomorphism class of $(Y_0^\bullet,\upsilon_0)$ (resp. $({Y'}^\bullet,\upsilon')$).

Following Definition \ref{def:pushout}, we now find homomorphisms 
$\lambda_0:P^{0,\bullet}\to \mathrm{Trunc}_{-p_0}(Y_0^\bullet)$
and $\lambda':P^{0,\bullet}\to \mathrm{Trunc}_{-p_0}({Y'}^\bullet)$ in $C^-(B')$ such that 
$\mathrm{Trunc}_{-p_0}(a_{Y_0})\circ \lambda_0=\epsilon=
\mathrm{Trunc}_{-p_0}(a_{Y'})\circ\lambda'$. Namely,
$$\lambda_0^0=0,\quad \lambda_0^{-j} = (d_{Y_0}^{-j}-c_Y^{-j},1)\; (2\le j\le p_0-1),
\quad \lambda_0^{-p_0} = e_0$$
where $e_0:Y^{-p_0}\to Y^{-p_0}/\mathrm{Image}(d_{Y_0}^{-p_0-1})$ is the natural surjection.
Similarly, we define $\lambda'$ by replacing $d_{Y_0}$ by $d_{Y'}$ and $e_0$ by
the natural surjection $e':Y^{-p_0}\to Y^{-p_0}/\mathrm{Image}(d_{Y'}^{-p_0-1})$.
Letting $\tilde{\lambda_0}$ (resp. $\tilde{\lambda'}$) be the restriction of $\lambda_0$
(resp. $\lambda'$) to $T^\bullet$, we obtain by using triangle diagrams in $D^-(B')$
similarly to $(\ref{eq:triangle})$ that 
\begin{equation}
\label{eq:betteretas}
\eta_{\xi_0} = \iota [1]\circ \tilde{\lambda_0}[1]\circ \eta_T\quad\mbox{and} \quad
\eta_{\xi'} = \iota [1]\circ \tilde{\lambda'}[1]\circ \eta_T
\end{equation}
in $D^-(B')$, 
where $\eta_T$ is the connecting homomorphism in the top row of $(\ref{eq:triangle})$.
Using the explicit computations of $T^\bullet$, $D^\bullet$ and
$\tau:A\hat{\otimes}_{A'}T^\bullet\to D^\bullet$ as before, one sees that there exists a
quasi-isomorphism $s_D:\tilde{Z}^\bullet[-1]\to D^\bullet$ in $C^-(B)$,
which is independent of the local isomorphism classes of $(Y_0^\bullet,\upsilon_0)$ and
$({Y'}^\bullet,\upsilon')$,  such that
$$\tilde{\lambda'}-\tilde{\lambda_0} = s_J\circ \tilde{\beta}_{Y'}[-1]\circ 
(s_D)^{-1}\circ (\tau\circ a_T)$$
in $D^-(B')$ where $s_J$ is as in $(\ref{eq:sJ})$, $\tilde{\beta}_{Y'}$ is as in $(\ref{eq:principali})$
and $a_T:T^\bullet\to A\hat{\otimes}_{A'}T^\bullet$ is the natural surjection. Note that
$\tau\circ a_T:T^\bullet \to D^\bullet$ is a quasi-isomorphism in $C^-(B')$. 
Hence
$$\eta_{\xi'}-\eta_{\xi_0} =  \iota [1]\circ (\tilde{\lambda_0}-\tilde{\lambda'})[1]\circ \eta_T
=(\iota\circ s_J)[1]\circ \tilde{\beta}_{Y'}\circ \left( (s_D)^{-1}[1]\circ (\tau\circ a_T)[1]\circ \eta_T\right)$$
in $D^-(B')$, completing the proof of Proposition \ref{prop:compare}.

%%%%%%%%%%%%%%%%%%%%%%%%%%%%%%%%%%%%%%%%%%%%%%%%%%%%%%%%%%%%%%%%%%%%%%%%%%%
%% Quotients by pro-$\ell'$ groups 
%%%%%%%%%%%%%%%%%%%%%%%%%%%%%%%%%%%%%%%%%%%%%%%%%%%%%%%%%%%%%%%%%%%%%%%%%%%

\section{Quotients by pro-$\ell'$ groups}
\label{s:prop}
\setcounter{equation}{0}

In this section, we give an application of the obstructions to lifting quasi-lifts as
determined in \S\ref{s:obstruct}.
As we have assumed throughout this paper, the field $k$ has positive characteristic $\ell$, 
and $V^\bullet$ is
a complex in $D^-(k[[G]])$ that has only finitely many non-zero cohomology groups, all of which
have finite $k$-dimension. Without loss of generality, we may assume that $\HH^i(V^\bullet)=0$
unless $-p_0\leq i \leq -1$.

\begin{rem}
\label{rem:prop}
Suppose there is a short exact sequence of profinite groups
\begin{equation}
\label{eq:sigh1}
1\to K\to G \to \Delta \to 1,
\end{equation}
where $K$ is a closed normal subgroup which is a pro-$\ell'$ group, i.e. the projective limit
of finite groups that have order prime to $\ell$. Let $R$ be an object in $\hat{\mathcal{C}}$, 
and suppose $M$ is a  projective
pseudocompact $R[[\Delta]]$-module. Then the inflation
$\mathrm{Inf}_{\Delta}^G\,M$ is a projective pseudocompact $R[[G]]$-module.
\end{rem}

\begin{prop}
\label{prop:prop}
Suppose $G$ and $\Delta$ are as in Remark $\ref{rem:prop}$,
$G$ has finite pseudocompact cohomology, and $V^\bullet$ is isomorphic to the inflation
$\mathrm{Inf}_\Delta^G\,V_\Delta^\bullet$ of a bounded above complex
$V_\Delta^\bullet$ of pseudocompact $k[[\Delta]]$-modules. Then 
 the two deformation functors 
$\hat{F}^G=\hat{F}^G_{V^\bullet}$ and $\hat{F}^\Delta=\hat{F}^\Delta_{V_\Delta^\bullet}$ which are defined according
to Definition $\ref{def:functordef}$ are naturally isomorphic. In consequence,
 $R(G,V^\bullet)\cong R(\Delta,V_\Delta^\bullet)$
and $(U(G,V^\bullet),\phi_U)\cong
(\mathrm{Inf}_{\Delta}^G\,U(\Delta,V_\Delta^\bullet),\mathrm{Inf}_{\Delta}^G\,\phi_U)$.
\end{prop}

\begin{proof}
It follows from the definition of finite pseudocompact cohomology (see Definition
\ref{dfn:bndcoh}) and from Remark \ref{rem:prop} that $\Delta$ also has finite pseudocompact
cohomology. It will be enough to show that the two deformation functors 
$\hat{F}^G=\hat{F}^G_{V^\bullet}$ and $\hat{F}^\Delta=\hat{F}^\Delta_{V_\Delta^\bullet}$
are naturally isomorphic. 

Let $0\to J\to A'\to A\to 0$ be an extension of objects $A',A$ in $\hat{\mathcal{C}}$ with 
$J^2=0$, and let $(Z_\Delta^\bullet,\zeta_\Delta)$ be a quasi-lift of $V_\Delta^\bullet$ over $A$. 
By Theorem \ref{thm:derivedresult}, we may assume that the terms of $Z_\Delta^\bullet$
are  projective pseudocompact $A[[\Delta]]$-modules. Hence
$(Z^\bullet,\zeta)=(\mathrm{Inf}_\Delta^G\,Z_\Delta^\bullet,\mathrm{Inf}_\Delta^G\,\zeta_\Delta)$ is
a quasi-lift of $V^\bullet$ over $A$, and by Remark \ref{rem:prop} the terms of
$Z^\bullet$ are projective pseudocompact $A[[G]]$-modules. 
By Remark \ref{rem:dumbdumb}, we can truncate
$Z_\Delta^\bullet$, and hence $Z^\bullet=\mathrm{Inf}_\Delta^G\,Z_\Delta^\bullet$, so as
to be able to assume Hypothesis \ref{hypo:obstruct} for both $Z_\Delta^\bullet$ and $Z^\bullet$.
Moreover, in view of Remark \ref{rem:prop}, we can choose the projective resolutions 
$P^{\bullet,\bullet}\to Z^\bullet\to 0$ and 
$P_\Delta^{\bullet,\bullet}\to Z_\Delta^\bullet\to 0$ in Definition \ref{def:spectralseqA}
such that $P^{\bullet,\bullet}=\mathrm{Inf}_\Delta^G\, P_\Delta^{\bullet,\bullet}$  and
such that $P_\Delta^{0,\bullet}$, and hence $P^{0,\bullet}$, is acyclic. 
We can also arrange that the projective Cartan-Eilenberg resolutions $M^{\bullet,\bullet,\bullet}$
and $M_\Delta^{\bullet,\bullet,\bullet}$ in Definition \ref{def:spectralseqB} satisfy
$M^{\bullet,\bullet,\bullet}=\mathrm{Inf}_\Delta^G\, M_\Delta^{\bullet,\bullet,\bullet}$.
Following the definition of $(\ref{eq:lowdegree0})$ and $(\ref{eq:lowdegree})$, we see 
that the natural inflation homomorphisms from $\Delta$ to $G$ identify
the sequences of low degree terms for $G$ and for $\Delta$.
Using Theorem \ref{thm:obstructions}(i), it follows that the obstruction to lifting 
$(Z_\Delta^\bullet,\zeta_\Delta)$ over $A'$ vanishes if and only if 
the obstruction to lifting $(Z^\bullet,\zeta)$ over $A'$ vanishes. 
Using Theorem \ref{thm:obstructions}(ii), we see that  if these obstructions vanish, then
the set of all local isomorphism classes of quasi-lifts of $(Z^\bullet,\zeta)$ over $A'$ is in bijection
with the set of all local isomorphism classes of quasi-lifts of $(Z_\Delta^\bullet,\zeta_\Delta)$ over $A'$.
This implies Proposition \ref{prop:prop}.
\end{proof}

%%%%%%%%%%%%%%%%%%%%%%%%%%%%%%%%%%%%%%%%%%%%%%%%%%%%%%%%%%%%%%%%%%%%%%%%%%%
%% bibliography
%%%%%%%%%%%%%%%%%%%%%%%%%%%%%%%%%%%%%%%%%%%%%%%%%%%%%%%%%%%%%%%%%%%%%%%%%%%

\end{document}